\documentclass[12pt,letterpaper]{amsart}

\usepackage{euscript,amsfonts,amssymb,amsmath,amscd, amsthm,enumerate}

\usepackage{tikz,bm,float}
\usepackage{physics}
\usepackage{mathtools}
\usetikzlibrary{calc}
\usetikzlibrary{patterns}
\colorlet{lgray}{white!85!black}
\colorlet{lred}{white!75!red}
\usepackage{xcolor}
\colorlet{NextBlue}{blue!60}

\usepackage{comment}
\DeclareRobustCommand{\frac}[3][0pt]{%
	{\begingroup\hspace{#1}#2\hspace{#1}\endgroup\over\hspace{#1}#3\hspace{#1}}}

\usepackage{graphicx}

\usepackage{color}

\usepackage{ytableau,tikz,varwidth}
\newcommand*\circled[1]{\tikz[baseline=(char.base)]{
		\node[shape=circle,draw,inner sep=2pt] (char) {#1};}}

\newcommand{\verteq}{\rotatebox{90}{$\,=$}}
\newcommand{\equalto}[2]{\underset{\scriptstyle\overset{\mkern4mu\verteq}{#2}}{#1}}

\let\emptyset\varnothing

\usepackage[font=scriptsize,labelfont=bf]{caption} 
\usepackage{caption}
\usepackage[hidelinks]{hyperref}

\usepackage[margin=0.9in]{geometry}
\newtheorem{theorem}{Theorem} 
\newtheorem*{theorem*}{Theorem}
\newtheorem{lemma}[theorem]{Lemma}
\newtheorem{definition}[theorem]{Definition}
\newtheorem{proposition}[theorem]{Proposition}

\newtheorem{corollary}[theorem]{Corollary}

\newtheorem{claim}[theorem]{Claim}

\theoremstyle{remark}

\numberwithin{equation}{section} \numberwithin{theorem}{section}

\usepackage{MnSymbol}

\usepackage{skak}

\usepackage{adjustbox}
\usetikzlibrary{arrows}
\usetikzlibrary{decorations.pathreplacing}

\title[Factorization of certain Macdonald Littlewood-Richardson coefficients]
{Factorization of certain Macdonald Littlewood-Richardson coefficients}

\author{Konstantin Matveev and Yuchen Wei}

\address[Konstantin Matveev]{Department of Mathematics, Rutgers University, New Brunswick, NJ, USA. E-mail: kosmatveev@gmail.com}

\address[Yuchen Wei]{Department of Mathematics, Rutgers University, New Brunswick, NJ, USA. E-mail: yuccwei@gmail.com}

\begin{document}

\begin{abstract}
We find and prove a factorization formula for certain Macdonald Littlewood-Richardson coefficients $c_{\lambda\mu}^{\nu}(q,t)$. Namely, we consider the case that the Kostka number $K_{\mu, \nu -\lambda}$ is $1$. This settles a particular case of a more general conjecture of Richard Stanely, \cite{St89}.  This conjecture proposes that a factorization formula exists whenever the corresponding regular Littlewood-Richardson coefficient $c^{\nu}_{\lambda \mu}$ is $1$. 
\end{abstract}

\maketitle

\tableofcontents

\section{Introduction}
\subsection{Preliminaries}
\subsubsection{Macdonald functions}
$P_{\lambda} = P_{\lambda}(q, t)$ is a Macdonald symmetric function,  \cite{Mac99}. $\{P_{\lambda}\}$ with $\lambda$ ranging over partitions is a basis of the algebra of symmetric functions over $\mathbb{C}(q, t)$. $P_{\lambda}P_{\mu} = \sum_{\nu} c^{\nu}_{\lambda \mu} (q, t)P_{\nu}$, where the sum ranges over partitions $\nu$ with $|\nu| = |\lambda| + |\mu|$. We are interested in studying these structure constants $c_{\lambda\mu}^{\nu}(q,t)$, i.e. Macdonald Littlewood-Richardson coefficients. 

\subsubsection{Schur functions}  $P_{\lambda}$ for $q = t$ becomes the Schur function  $S_{\lambda}$, which doesn't depend on $q = t$.  For $n \geq \ell(\lambda)$ polynomial $S_{\lambda}(x_{1}, x_{2}, \ldots, x_{n})$ can be viewed as the character of the Schur module $E^{\lambda}$, an irreducible representation of $GL_{n}(\mathbb{C})$. So for $n \geq \ell(\nu) \geq \ell(\lambda), \ell(\mu)$ the coefficient $c^{\nu}_{\lambda \mu} =  c^{\nu}_{\lambda \mu}(q, q)$ is a non-negative integer that counts the multiplicity of $E^{\nu}$ in $E^{\lambda} \otimes E^{\mu}$,  \cite[chapter 8]{F97}. 

\subsubsection{Littlewood-Richardson rule} 
$c^{\nu}_{\lambda \mu}$ counts the number of semistandard skew tableaux $T$ of shape $\nu/\lambda$ with weight $\mu$ and additional property  that $w(T)$ is a lattice word. Here word $w(T)$ is obtained by reading symbols in $T$ from right to left in successive rows, starting with the top row. A word of length $N$ in the symbols $1,2,...,n $ is said to be a lattice word if for $1 \le r\le N $ and $ 1\le i\le n-1 $, the number of occurrences of the symbol $ i $ in the first $r$ letters of $w$ is not less than the number of occurrences of $ i+1$. This is the celebrated Littlewood-Richardson rule, \cite{LR34}, \cite[chapter 5]{F97}, \cite{Ste02}. $c^{\nu}_{\lambda \mu}$ also enumerates other families of combinatorial objects: Berenstein-Zelevinsky patterns,  \cite{BZ92}, Knutson-Tao  honeycombs/puzzles/hives,  \cite{KT99}, \cite{KTW04},  \cite{B00}, Vakil checkergames,  \cite{V06}. 

\subsubsection{Pieri formulas}
For one-row $\mu = (r)$ or one-column $\mu =1^{r}$ the corresponding Macdonald Littlewood-Richardson coefficients are given explicitly by the Pieri formulas.
	 	\begin{proposition}[{Mac99}]
	 	\label{prop: Pieri}
		\begin{align}
		\label{VerticalPieri}
				P_\lambda\cdot P_{1^{r}}=\sum_{\lambda\prec_v \nu,\ |\nu|-|\lambda|=r}\psi_{\nu/\lambda}'P_\nu,
		\end{align}
		where $ \nu/\lambda $ is a skew shape of $ r $ boxes without any two boxes in the same row, i.e., a vertical strip,  and $$ \psi'_{\nu/\lambda}:=\prod_{s\in C_{\nu/\lambda}-R_{\nu/\lambda}}\frac{b_\nu(s)}{b_{\lambda}(s)},\quad b_{\lambda}(s)=b_{\lambda}(i,j):=\begin{cases}\frac{1-q^{\lambda_i-j}t^{\lambda_j'-i+1}}{1-q^{\lambda_i-j+1}t^{\lambda_j'-i}} & \text{ if }s\in \lambda\\
		1 & \text{ otherwise.}
		\end{cases}, $$
		where $ C_{\nu/\lambda}$ (respectively $R_{\nu/\lambda}) $ is the set of columns (respectively rows) in $ \nu $ that intersect $ \nu/\lambda $, and $ \lambda' $ denotes the conjugate partition of $ \lambda $. Also, 
		\begin{align}
		\label{HorizontalPieri}
				P_\lambda\cdot P_{(r)}=\sum_{\lambda\prec_h \nu,\ |\nu|-|\lambda|=r} \frac{(q; q)_{r}}{(t; q)_{r}}\varphi_{\nu/\lambda}P_\nu, 
			\end{align}
				where $ \nu/\lambda $ is a skew shape of $ r $ boxes without any two boxes in the same column, i.e., a horizontal strip,  and
				$$ \varphi_{\nu/\lambda}:=\prod_{s\in C_{\nu/\lambda}}\frac{b_\nu(s)}{b_{\lambda}(s)}, \quad (a; q)_{s}:= \prod_{i=1}^{s} (1-aq^{i-1}).$$			
	 	\end{proposition}
	 	
See \cite[VI.6]{Mac99} for more details. Note that $P_{1^{r}} = e_{r}$, the $r$-th elementary symmetric function. 	It follows from the Pieri formulas that the corresponding Macdonald Littlewoood-Richardson coeffiecients are positive for $-1 < q, t < 1$ and become $1$ for $q=t$, the Schur case.  

\subsubsection{Formulas for Macdonald Littlewood-Richardson coefficients} \cite{S11} uses interpolation Macdonald polynomials to find an expression for $c^{\nu}_{\lambda \mu}(q, t)$, \cite{Y12} uses nonsymmetric Macdonald polynomials and combinatorics of alcove walks to find another formula. However, both expressions are very complex. We are interested in finding special cases when the corresponding coefficients easily factorize. 
More precisely, we are looking for triples of partitions $(\lambda, \mu, \nu)$ such that $c^{\nu}_{\lambda \mu}(q, t)$ is a product of terms $\left(\frac{1-q^{a}t^{b+1}}{1-q^{a+1}t^{b}} \right)^{\pm 1}$ with $a, b \in \mathbb{Z}$. Pieri formulas show that $(\lambda, 1^{r}, \nu)$ are such triples whenever $\nu/\lambda$ is a vertical $r$-strip, $(\lambda, (r), \nu)$ are such triples whenever $\nu/\lambda$ is a horizontal $r$-strip. By commutativity, same holds for $(1^{r}, \mu, \nu)$ and 
$((r), \mu, \nu)$. 

\subsubsection{Stanley's conjecture}
Setting $q=t$ implies that triples with factorization of this kind must satisfy $c^{\nu}_{\lambda \mu} = 1$. Richard Stanley conjectured in \cite[Conjecture 8.5]{St89} that the converse is also true. More precisely, his conjecture was for Jack functions which can be obtained from Macdonald functions by setting $q = t^{\alpha}$ and $t \to 1$. But it transfers to the more general Macdonald case in a straightforward way.   As far as we know, the conjecture in general remains open. Stanley also suggested that the factorization formula must belong to a specified finite family. However, it remains unclear how in general to chose the correct element of this family.  \cite{N16} proved Stanley's conjecture for the special case $\ell(\lambda), \ell(\mu), \ell(\nu) \leq 3$. In this paper we identify another family of special cases for which we can prove factorization.

\subsection{Results}
\subsubsection{Kostka numbers} The Kostka number $K_{\mu \chi}$ for a partition $\mu$ and a composition $\chi$ is the number of semistandard Young tableaux of shape $\mu$ and weight $\chi$. Suppose that $ \lambda=(\lambda_1,\lambda_2,...,\lambda_n) $, where some of the $\lambda_{i}$'s might be zeros. 
$\ell(\mu), \ell(\nu) \leq n$, $ \nu=(\lambda_1+\chi_1, \lambda_2+\chi_2, \ldots, \lambda_n+\chi_n) $, where all $\chi_i \ge 0 $ and $ \sum_{i=1}^n \chi_i=|\mu|$. We start with the following corollary of the Littlewood-Richardson rule. 
\begin{proposition}
\label{prop:LRvsKostka}
$c^{\nu}_{\lambda \mu} \leq K_{\mu \chi}$ with equality, in particular, when $\nu/\lambda$ is a horizontal strip.
\end{proposition}
See section \ref{sec:UST} for the proof of proposition \ref{prop:LRvsKostka}. Note that $\lambda_i-\lambda_{i+1}\ge \ell(\mu') $ for any $ 1\le i\le n-1$ guarantees that $\nu/\lambda$ is a horizontal strip. 
Thus for $\lambda$ with large row gaps the Littlewood-Richardson coefficient turns into the corresponding Kostka number.
\subsubsection{Unique semistandard tableaux}
\cite{BZ92} explicitly  describes pairs $(\mu, \chi)$ with $K_{\mu \chi} =1$. We prove a related result

\begin{proposition}
\label{prop: UniqueSST}
Given a semistandard tableau $ T $ with $ n $ boxes, the following are equivalent:
\begin{itemize} 
\item[($C_{n}$):] $ T $ is unique of its shape and weight.
			
\item[($C'_{n}$):]  For any two columns in $ T $, the later one is either obtained from the former by changing at most one value, or is a subset of the former.
\end{itemize} 
\end{proposition}		
\begin{figure}[h]
  \includegraphics[width=0.5\linewidth]{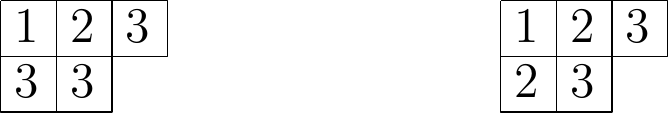}
  \caption{Left: Unique tableau of its shape and weight. Right: Non-unique tableau, a third column is not contained in the first column.}
  \label{uniquevsnot}
\end{figure}
See section \ref{sec:UST} for the proof of proposition \ref{prop: UniqueSST}.
\subsubsection{Macdonald Littlewood-Richardson coefficient as a rational function}
Let $X_{i}:=q^{\lambda_{i}}t^{1-i}$. Then the Pieri multiplicity 
$\psi'_{\nu/\lambda}$ can be expressed as 
\begin{align}
\label{VariableVerticalPieri}
\prod_{1 \leq j<k \leq n: \ \nu_{j} = \lambda_{j}, \ \nu_{k} = \lambda_{k} + 1}  \frac{X_{k} - X_{j} tq^{-1}}{X_{k} - X_{j}} \cdot \frac{X_{k} - X_{j} t^{-1}}{X_{k} - X_{j} q^{-1}}.
\end{align}
Note that the denominator of \eqref{VariableVerticalPieri} doesn't vanish for general $q, t$ and any composition $\lambda$. If $\lambda$ is a partition, then the numerator of \eqref{VariableVerticalPieri} for general $q, t$ vanishes precisely when $\nu$ fails to be a partition. 
Now we can express $P_{\mu}$ as a polynomial in $e_{r}$'s and repeatedly use \eqref{VariableVerticalPieri} to express $c^{\nu}_{\lambda \mu}(q, t)$ as a rational function in $X_{1}, X_{2}, \ldots, X_{n}$. This approach breaks the symmetry between $\lambda$ and $\mu$, but allows to study all $\lambda$'s at the same time by means of the same rational function. We originally encountered the folrmulas \eqref{FactorizationFormula} below by studying such rational functions in \emph{Mathematica}.

\subsubsection{Factorization formula} Let $T$ be a unique semistandard tableau with entries from $1, 2,$ $\ldots, n$. Call a triple $ (j,k,m) $ \textit{admissible} if $ 1\le j<k\le n $, and the $ m $-th column of $ T $ contains $ k $ but does not contain $ j $. Denote by $ a(j,k,m)$ the number of columns of $ T $ among the first $ (m-1) $ columns that also contain $ k $, but do not contain $ j $. Denote by $ b(j,k,m)$ the number of columns of $ T $ among the first $ (m-1) $ columns that contain $ j $, but do not contain $ k $. Denote by $ \psi_T(q,t) $ the $ \psi $-weight of the tableau $ T $ as in the formula $ P_\mu(x)=\sum_{T}\psi_T(q,t)x^T $. Our main result is the following. 
\begin{theorem}
\label{theorem: FactorizationFormula}

Suppose that $ K_{\mu \chi}=1 $. Let $T$ be the corresponding unique semistandard tableau. Then the Macdonald Littlewood-Richardson coefficient $ c_{\lambda\mu}^\nu(q,t) $ is given by
		
		\begin{align}
		\label{FactorizationFormula}
			c^{\nu}_{\lambda\mu}(q,t)=\psi_T(q,t)  \cdot \prod_{(j,k,m) \ -\text{admissible}} \left(\frac{X_{k} - q^{-a(j,k,m)+b(j, k,m)-1} tX_{j}}{X_{k}-q^{-a(j, k, m)+b(j, k, m)} X_{j}}\cdot
			\frac{X_{k} - q^{-a(j,k,m)}t^{-1} X_{j}}{X_{k}-q^{-a(j,k,m)-1}X_{j}} \right), 
		\end{align}
		where $ X_i:=q^{\lambda_i}t^{1-i} $.
		\end{theorem}

\subsubsection{Example} If $T$ is a one-column tableau, it is always unique of its shape and weight. Then $a(j, k, 1) = b(j, k, 1)=0$, so \eqref{FactorizationFormula} becomes the vertical Pieri formula \eqref{VariableVerticalPieri}.  If $T$ is a one-row tableau, it is also always unique of its shape and weight. This case corresponds to the horizontal Pieri-formula, see subsection \ref{sub:horizontal} for more details. Here is an example that is not a Pieri formula. Let $n=3$, $\mu = (3, 2)$, $\nu_{1}-\lambda_{1}=\nu_{2}-\lambda_{2}=1, \nu_{3}-\lambda_{3}=3$. The corresponding unique semistandard tableau is shown on the left of Figure \ref{uniquevsnot}. The formula \eqref{FactorizationFormula} in this case becomes
\begin{multline*}
\frac{(1+ q) (1 - t)}{1 -q t} \cdot 
\frac{\left(X_3-q^{-1}tX_2\right)  \left(X_3-t^{-1} X_2\right)}{
 \left(X_3- X_2\right) \left(X_3-q^{-1} X_2\right)} \cdot \frac{\left(X_2-tX_1\right)  \left(X_2-t^{-1} X_2\right)}{
 \left(X_2- qX_1\right) \left(X_2-q^{-1} X_1\right)} \\ \cdot \frac{\left(X_3-q^{-1}tX_1\right)  \left(X_3-t^{-1} X_1\right)}{
 \left(X_3- X_1\right) \left(X_3-q^{-1} X_1\right)} \cdot \frac{\left(X_3-q^{-2}tX_1\right)  \left(X_3-q^{-1}t^{-1} X_1\right)}{\left(X_3- q^{-1}X_1\right) \left(X_3-q^{-2} X_1\right)} \cdot \frac{\left(X_3-q^{-1}tX_2\right)  \left(X_3-t^{-1} X_2\right)}{
 \left(X_3- X_2\right) \left(X_3-q^{-1} X_2\right)} 
\end{multline*}
		
\subsubsection{Connection with Stanley's conjecture} Let $$ U(a,l)=1-q^{a+1}t^{l} ,\quad L(a,l)=1-q^{a}t^{l+1} $$ denote the upper hook length and the lower hook length, respectively. Let $ s $ denote a box in a partition, and let $ a(s) $(respectively, $ l(s) $) be the arm(respectively, leg) length at $ s $.  We show that	

\begin{proposition}\label{prop:integral}
Under the assumption of Theorem \ref{theorem: FactorizationFormula}, we have
	\begin{align}\label{eq:integral}
	\left( \prod_{s\in\lambda}L(a(s),l(s)) \right)\cdot\left( \prod_{s\in\mu}L(a(s),l(s)) \right)\cdot\left( \prod_{s\in\nu}U(a(s),l(s)) \right)\cdot c_{\lambda\mu}^{\nu}(q,t)  
	\end{align}
	is a polynomial in $ L, U $, moreover, the number of factors in the form of $ U $ equals the number of factors in the form of $ L $.
\end{proposition}

See section \ref{section:connection} for the proof of proposition \ref{prop:integral}.

Note that Stanley's conjecture is about the Jack polynomials and requires the weaker assumption that $ c_{\lambda\mu}^\nu=1 $. However, the Jack case can be viewed as a specialization of the Macdonald case by letting $ q=t^\alpha $ and $ t\to 1 $. If we apply the L'H\^opital's rule in the limit of $ c_{\lambda\mu}^{\nu}(t^\alpha,t) $ as $ t\to 1 $, and replace $ U(a,l) $(respectively, $ L(a,l) $) by $ \alpha(a+1)+l $(respectively, $ \alpha(a)+l+1 $), then we get the Stanley's conjecture.

The proof of proposition \ref{prop:integral} follows directly from theorem \ref{theorem: FactorizationFormula}. The idea is to rewrite factors in \ref{theorem: FactorizationFormula} into the form of lower and upper hook lengths and cancel the repeated factors.

\subsubsection{Example.} 
\begin{center}
	\begin{tabular}{ccc}	
		
		\noindent\begin{minipage}{0.3\linewidth}
			\centering	
			\ytableausetup{mathmode,boxsize=1.3em}
			 	\begin{ytableau}
					\ &\ & \ &1&1\\
					\ & \ &1&2\\
					\ &2&2&3\\
				    \ &3&3
			\end{ytableau}
		\end{minipage}
		
		$\quad \iff$

		\noindent\begin{minipage}{0.3\linewidth}
			\centering	
			\ytableausetup{mathmode,boxsize=1.3em}

				\begin{ytableau}
					1&1&2\\
					2&3&3\\
					3&4&4
			\end{ytableau}
		\end{minipage}
	\end{tabular}
	\renewcommand{\figurename}{Bijection}
	\captionof*{figure}{An example with $\lambda=(3,2,1,1), \mu=(3,3,3), \nu=(5,4,4,3)$.}
\end{center}

Consider the above partitions and tableaux. Theorem \ref{theorem: FactorizationFormula} tells us that 

\begin{align*} c_{\lambda\mu}^\nu(q,t)=\underbrace{\left(\frac{L(0,0)}{U(0,0)}\cdot\frac{U(1,0)}{L(1,0)}\cdot\frac{L(0,0)}{U(0,0)}\cdot\frac{U(2,0)}{L(2,0)}\right)}_{\psi_{T}(q,t)}&\cdot
\left(\textcolor{red}{ \frac{L(\lambda_{2}-\lambda_{3}-1,1)}{U(\lambda_{2}-\lambda_{3}-1,1)}}\cdot
\frac{U(\lambda_{2}-\lambda_{3}-1,0)}{L(\lambda_{2}-\lambda_{3}-1,0)} \right)\\
\cdot\left(\frac{L(\lambda_{2}-\lambda_{4},2)}{U(\lambda_{2}-\lambda_{4},2)}\cdot
\textcolor{red}{\frac{U(\lambda_{2}-\lambda_{4}-1,1)}{L(\lambda_{2}-\lambda_{4}-1,1)}} \right)&\cdot
\left(\textcolor{blue}{\frac{L(\lambda_{1}-\lambda_{2},1)}{U(\lambda_{1}-\lambda_{2},1)}}\cdot
\frac{U(\lambda_{1}-\lambda_{2}-1,0)}{L(\lambda_{1}-\lambda_{2}-1,0)} \right)\\
\cdot
\left(\textcolor{green}{\frac{L(\lambda_{1}-\lambda_{3}-1,2)}{U(\lambda_{1}-\lambda_{3}-1,2)}}\cdot
\textcolor{blue}{\frac{U(\lambda_{1}-\lambda_{3}-1,1)}{L(\lambda_{1}-\lambda_{3}-1,1)}} \right)&\cdot
\left(\frac{L(\lambda_1-\lambda_4,3)}{U(\lambda_1-\lambda_4,3)}\cdot\textcolor{green}{\frac{U(\lambda_1-\lambda_4-1,2)}{L(\lambda_1-\lambda_4-1,2)}}\right)\\
\end{align*}
where the fractions with the same color are reciprocals of each other. Now we can give a possible way of realizing $ c_{\lambda\mu}^\nu(q,t) $ as a product of upper and lower hook lengths as follows

\begin{center}
	\begin{tabular}{ccc}	
		
		\noindent\begin{minipage}{0.3\linewidth}
			\centering	
			\ytableausetup{mathmode,boxsize=1.3em}
		 	\begin{ytableau}
					L &L &\color{orange} U\\
					L & \color{orange}U\\
					L\\
					L
			\end{ytableau}
		\end{minipage}
		
		$\quad $

		\noindent\begin{minipage}{0.3\linewidth}
			\centering	
			\ytableausetup{mathmode,boxsize=1.3em}

				\begin{ytableau}
					L &L & L\\
					L &L & L\\
					\color{orange} U &\color{orange} U &L
			\end{ytableau}
		\end{minipage}
	
	$\quad $

	\noindent\begin{minipage}{0.3\linewidth}
		\centering	
		\ytableausetup{mathmode,boxsize=1.3em}
		
			\begin{ytableau}
				U & U & \color{orange}L & U & \color{orange}L\\
				U & U &\color{orange} L & U\\
				U & U & U & \color{orange}L\\
			    U & U & U\\
		\end{ytableau}
	\end{minipage}
	\end{tabular}
	\renewcommand{\figurename}{Bijection}
	\captionof*{figure}{This picture shows how to flip lower hooks in $ \lambda, \mu $ and upper hooks in $ \nu $. }
\end{center}

Note that the fraction $ \textcolor{red}{\frac{U(\lambda_{2}-\lambda_{4}-1,1)}{L(\lambda_{2}-\lambda_{4}-1,1)}} $ corresponds to a ``fictitious flipper"(to be discussed in section \ref{section:connection}) in the sense that there is no box in the second row of $ \lambda $ with leg length one, however, we can pair it with its reciprocal $ \textcolor{red}{\frac{L(\lambda_{2}-\lambda_{3}-1,1)}{U(\lambda_{2}-\lambda_{3}-1,1)}} $.

The fraction $ \textcolor{blue}{\frac{L(\lambda_{1}-\lambda_{2},1)}{U(\lambda_{1}-\lambda_{2},1)}} $ correspond to a ``fictitious flipper" since there is no box in the first row of $ \nu $ with leg length one, however, we pair it with its reciprocal $ \textcolor{blue}{\frac{U(\lambda_{1}-\lambda_{3}-1,1)}{L(\lambda_{1}-\lambda_{3}-1,1)}} $.

Likewise, the fraction $ \textcolor{green}{\frac{U(\lambda_1-\lambda_4-1,2)}{L(\lambda_1-\lambda_4-1,2)}} $ is a ``fictitious flipper" since there is no box in the first row of $ \lambda $ with leg length two, however, we pair it with its reciprocal $ \textcolor{green}{\frac{L(\lambda_{1}-\lambda_{3}-1,2)}{U(\lambda_{1}-\lambda_{3}-1,2)}} $.

All the section \ref{section:connection} is devoted to showing that whenever we have a ``fictitious flipper", it is guaranteed to find its reciprocal in the expression of $ c_{\lambda\mu}^{\nu}(q,t) $. The way that we examine this fact is by comparing for each fixed triple of integers $ (i,j-1,j) $ such that $ i<j-1 $, the number of ``fictitious flippers" in the form of $ \frac{U}{L} $ with leg length $ (j-i)-1 $(respectively, $ \frac{L}{U} $ with leg length $ (j-1)-i $) is no more than the number of their reciprocals with leg length $ (j-1)-i $(respectively, $ (j-i)-1 $).

\begin{center}
	\begin{tabular}{ccc}	
		
			\centering	
			\ytableausetup{mathmode,boxsize=1.3em}
			 	\begin{ytableau}
				\none[i]& \none&\none&\none& \none&\none&\none&\none & \ & \ & \ &\ &\ & \ & &\ \\
				\none& \none&\none&\none&\none& \none&\none&\none&\none[\dots]&\none[\dots]\\
				\none& \none&\none&	\none&\none& \none&\none&\none&\none[\dots]&\none[\dots]\\
				\none[j-1]& \none&\none&\none&\none& *(lime) &  *(lime) & *(lime) & *(lime) & *(lime) &  *(lime) &  &  \\
				\none[j]& \none&\none&\ &\ & \ & \ & \ & \ & \ & 
			\end{ytableau}
		
	\end{tabular}
\renewcommand{\figurename}{fictitious flippers}
\captionof*{figure}{This picture demenstrates that all the potential ``fictitious flippers" located in the $ i $-th row with ``fictitious leg length" $ (j-1)-i $ in $ \nu $, or $ (j-i)-1 $ in $ \lambda $, correspond to some of the colored boxes.}
\end{center}

\subsubsection{Proof outline of Proposition \ref{prop: UniqueSST}}

We do induction on the total number $ n $ of boxes in a given semistandard tableau $ T $.

To show that $ (C_n)\implies (C_n') $, for arbitrary two columns in $ T $, we identify seven different cases characterized by the relative positions of the lowest boxes in these two columns as well as the appearances of the largest integer $ l $ of $ T $ in these two columns. By the induction hypotheses, we can quickly go through the first five cases. 

When the sixth case happens, we use induction to argue that $ (C_n') $ must hold, namely, $ [\lambda_j'] $ is a subset of $ [\lambda_i'] $ in this particular case. It turns out that this further boils down to showing that when $ \lambda_i'=\lambda_j'+1 $, we must have $ [\lambda_j']=[\lambda_i']\setminus\{l\} $ for any $ i<j $ satisfying Case 6. Proceed by proving by contradiction. Assume on the contrary that $ \lambda_i=\lambda_j'+1 $ and $ [\lambda_j']\ne [\lambda_i']\setminus\{ l \} $ for some $ i<j $ in Case 6, then after doing a counterclockwise shift of a sequence of boxes(the notations are defined in the proof):
\begin{align*}
	( T(\lambda_j'-h,S), T(\lambda_j'-h+1,s), & T(\lambda_j'-h+2,s),..., T(\lambda_i',s),T(\lambda_i',s+1),...,\\
	& T(\lambda_i', m), T(\lambda_j',S),T(\lambda_j'-1,S),..., T(\lambda_j'-h+1, S) )
\end{align*}
we will get a new tableau. This contradicts $ (C_n) $.

It remains to show that Case 7 is impossible. After ruling out the simple case where $ T(\lambda_i',i)>T(\lambda_j'-1,j) $ we claim that doing a clockwise shift of the boxes:

$$ (T(\lambda_j'-h+1,i),T(\lambda_j'-(h-1)+1,i),...,T(\lambda_j',i), T(\lambda_j',j),T(\lambda_j'-1,j),...,T(\lambda_j'-h+1,j))  
$$
yields a new tableau, contradicting $ (C_n) $.

To show that $ (C_n')\implies (C_n) $, we prove by contradiction. Assume that we have two distinct tableaux with the same shape and weight, we first show that the weights of each corresponding vertical rectangular blocks must be the same. This implies that the positions of $ l $ are fixed. Then we observe that $ C_{n-\norm{L}}' $ holds which implies that the rest of the tableaux are also the same. This concludes the proof.

Suppose tableau $T$ is unique of its shape and weight. It is comprised of a sequence of rectangular blocks of decreasing heights. Each of these blocks must also be unique of its shape and weight by proposition \ref{prop: UniqueSST}. Consider such block of size $k \times m$. Then proposition \ref{prop: UniqueSST} implies that weights of its columns are sets of size $k$ and any two of these sets differ by at most one element. 
\begin{proposition}
\label{prop: threekinds}
Suppose $S_{1}, S_{2}, \ldots, S_{m}$ are sets of cardinality $k$ such that $|S_{i} \Delta S_{j}| \leq 1$ for all $1 \leq i, j \leq m$. 
Then either $S_{1} =S_{2} = \cdots = S_{m}$ (call it first kind), or $\displaystyle \biggr \lvert \bigcap_{i=1}^{m}S_{i} \biggr \rvert = k-1$ (call it second kind), or  $\displaystyle \biggr \lvert \bigcup_{i=1}^{m}S_{i} \biggr \rvert = k+1$ (call it third kind).  
\end{proposition}
See section \ref{sec:factorization} for the proof of the proposition \ref{prop: threekinds}. Note that some rectangular blocks can be both of the second kind and the third kind. One row tableau is either of the first kind, or of the second kind. The idea is to prove theorem \ref{theorem: FactorizationFormula} individually for rectangular blocks of each kind and then put everything together with the help of induction. The main tool is vertical Pieri formulas \ref{prop: Pieri}. 
\subsubsection{Acknowledgements} The authors would like to thank Siddhartha Sahi and Vadim Gorin for fruitful discussions connected to the subject of the paper.

\section{Unique semistandard tableau}
\label{sec:UST}
\begin{proof}[Proof of Proposition \ref{prop:LRvsKostka}]
	Consider a Littlewood-Richardson tableau $T$ of shape $\nu/\lambda$ and weight $\mu$. 
	For $1 \leq k \leq j \leq n$ let $\mu_{k}^{j}$ denote the number of entries equal to $k$ in the first $j$ rows of $T$. Then $\mu^{n}_{k} = \mu_{k}$.  Clearly, $\mu_{k}^{j+1} \geq \mu_{k}
	^{j}$ for $j \leq n-1$. Lattice word condition is equivalent to $\mu_{k+1}^{j+1} \leq \mu_{k}
	^{j}$ for $j \leq n-1$. $\sum_{k=1}^{j} \mu^{j}_{k} - \sum_{k=1}^{j-1} \mu^{j-1}_{k}$ is the total number of entries in the $j$-th row, so is $\nu_{j} - \lambda_{j}$. 
	So $\left\{\mu^{j}_{k}\right\}_{1 \leq k \leq j \leq n}$ is an interlacing Gelfand-Tsetlin array with top row $\mu$ and weight $\nu-\lambda$. It corresponds to a semistandard tableau with shape $\mu$ and weight $\nu-\lambda$. So we have constructed an injective map from $\{\text{Littlewood-Richardson tableau of shape $\nu/\lambda$ and weight $\mu$}\}$ to $\{\text{Semistandard tableau of shape $\mu$ and weight $\nu - \lambda$}\}$. Thus $c^{\nu}_{\lambda \mu} \leq K_{\mu, \nu-\lambda}$. This map in general is not a bijection, since there are additional constraints coming from columns of $T$ that we have not taken into account. However, if $\nu/\lambda$ is a horizontal strip, there are no such constraints. So in this case the map is a bijection and $c^{\nu}_{\lambda \mu} = K_{\mu, \nu-\lambda}$. 
\end{proof}
\begin{center}
	\begin{tabular}{ccc}	
		
		\noindent\begin{minipage}{0.3\linewidth}
			\centering	
			\ytableausetup{mathmode,boxsize=1.3em}
			{\tiny 	\begin{ytableau}
					\none&\none&\none&\none&\none&\none&\none&\none&\none&\none& *(NextBlue) 1&*(NextBlue) 1&*(NextBlue) 1&*(NextBlue) 1 \\
					\none&\none&\none&\none&\none&\none&*(NextBlue!50)1&*(NextBlue!50)2&*(NextBlue!50)2\\
					\none&\none&\none&*(NextBlue!20)2&*(NextBlue!20)2\\
					2&3&\none
			\end{ytableau}}
		\end{minipage}
		
		$\quad \iff$
		
		\noindent\begin{minipage}{0.3\linewidth}
			\centering	
			$${\footnotesize \begin{matrix}
					5&&5&&1&&0\\
					&5&&4&&0\\
					&&5&&2\\
					&&&4
			\end{matrix}}$$
		\end{minipage}
		
		$ \iff $
		
		\noindent\begin{minipage}{0.3\linewidth}
			\centering	
			\ytableausetup{mathmode,boxsize=1.3em}
			{\tiny
				\begin{ytableau}
					*(NextBlue)1&*(NextBlue)1&*(NextBlue)1&*(NextBlue)1&*(NextBlue!50)2\\
					*(NextBlue!50)2&*(NextBlue!50)2&*(NextBlue!20)3&*(NextBlue!20)3&4\\
					4
			\end{ytableau}}
		\end{minipage}
	\end{tabular}
	\renewcommand{\figurename}{Bijection}
	\captionof*{figure}{An example of the bijection when $ \lambda $ has $ 4 $ rows, $ \mu=(5,5,1) $, $ l=(4,3,2,2) $.}
\end{center}

		\begin{proof}[Proof of Proposition \ref{prop: UniqueSST}]
			\noindent(Notations: Let $ \lambda' $ denote the conjugate partition of a partition $ \lambda=(\lambda_1\ge\lambda_2\ge\cdots\ge\lambda_k>0) $, where $ k $ is the length of $ \lambda $. Let $ [\lambda_i'] $ denote the set of boxes in the $ i $-th column of $ \lambda' $. Since the column is filled with different integers, we identify the integers with their corresponding boxes in the same column. For example, if $ l $ is an integer appeared in the $ i $-th column, we write $ l\in[\lambda_i'] $. Let $ T(r,c) $ denote the box in the $ r $-th row and the $ c $-th column.) 
			
			Clearly, it is true if $ T $ has only one box. Proceed by doing induction on the number of boxes of $ T $ and assume that it holds for all integers less than $ n $.
			
			$(C_n)\implies (C_n'):$ Let $ l $ be the largest integer in $ T $. The boxes filled with $ l $ form a horizontal strip in $ T $. Let $ L $ denote this horizontal strip with $ \norm{L} $ boxes. It follows that $ \widetilde{T}:=T\setminus L $ is unique of its shape and weight.
			
			By induction, $ \widetilde{T} $ satisfies ($C_{n-\norm{L}}'$), we verify that $ T $ satisfies ($C_n'$). For any two columns in $ T $, there are seven cases, illustrated below, to consider:
			
			\vspace*{10pt}
			
			\begin{tabular}{cc}
				\noindent\begin{minipage}{.3\linewidth}		\centering	
					\ytableausetup{mathmode,boxsize=2em}
					\begin{ytableau}
						\space &\none[\dots] & \space \\
						\space &\none[\dots] & \space \\
						\none[\vdots] & \none[\vdots]& \none[\vdots] \\
						l &\none[\dots] & l \\
					\end{ytableau}
					\renewcommand{\figurename}{Case}
					\renewcommand{\thefigure}{1}
					\captionof{figure}{}
				\end{minipage} 
				
				\noindent\begin{minipage}{.3\linewidth}
					\centering
					\begin{ytableau}
						\space &\none[\dots] & \space \\
						\space &\none[\dots] & \space \\
						\none[\vdots] & \none[\vdots]
						& \none[\vdots] \\
						\space &\none[\dots] & l \\
					\end{ytableau}
					\renewcommand{\figurename}{Case}
					\renewcommand{\thefigure}{2}
					\captionof{figure}{}
				\end{minipage} 
				
				\noindent\begin{minipage}{.3\linewidth}
					\centering
					\begin{ytableau}
						\space &\none[\dots] & \space \\
						\space &\none[\dots] & \space \\
						\none[\vdots] & \none[\vdots] & \none[\vdots] \\
						\space &\none[\dots] & \space \\
					\end{ytableau}
					\renewcommand{\figurename}{Case}
					\renewcommand{\thefigure}{3}
					\captionof{figure}{}
				\end{minipage} 
				
			\end{tabular}
			
			\begin{center}
				\begin{tabular}{ccc}
					\noindent\begin{minipage}{.25\linewidth}
						\centering	
						\ytableausetup{mathmode,boxsize=2em}
						\begin{ytableau}
							\space &\none[\dots] & \space \\
							\space &\none[\dots] & \space \\
							\none[\vdots] & \none[\vdots] & \none[\vdots] \\
							\space &\none[\dots] & l \\
							\none[\vdots]& \none[\cdots]\\
							l &\none[\cdots]
						\end{ytableau}
						\renewcommand{\figurename}{Case}
						\renewcommand{\thefigure}{4}
						\captionof{figure}{}
					\end{minipage} 	
					
					\noindent\begin{minipage}{.25\linewidth}
						\centering	
						\ytableausetup{mathmode,boxsize=2em}
						\begin{ytableau}
							\space &\none[\dots] & \space \\
							\space &\none[\dots] & \space \\
							\none[\vdots] & \none[\vdots]
							& \none[\vdots] \\
							\space &\none[\dots] & \space \\
							\none[\vdots]& \none[\cdots]\\
							\space &\none[\cdots]
						\end{ytableau}
						\renewcommand{\figurename}{Case}
						\renewcommand{\thefigure}{5}
						\captionof{figure}{}
					\end{minipage} 
					
					\noindent\begin{minipage}{.25\linewidth}
						\centering	
						\ytableausetup{mathmode,boxsize=2em}
						\begin{ytableau}
							\space &\none[\dots] & \space \\
							\space &\none[\dots] & \space \\
							\none[\vdots] & \none[\vdots]
							& \none[\vdots] \\
							\space &\none[\dots] & \space \\
							\none[\vdots]& \none[\cdots]\\
							l &\none[\cdots]
						\end{ytableau}
						\renewcommand{\figurename}{Case}
						\renewcommand{\thefigure}{6}
						\captionof{figure}{}
					\end{minipage} 
					
					\noindent\begin{minipage}{.25\linewidth}
						\centering	
						\ytableausetup{mathmode,boxsize=2em}
						\begin{ytableau}
							\space &\none[\dots] & \space \\
							\space &\none[\dots] & \space \\
							\none[\vdots] & \none[\vdots]
							& \none[\vdots] \\
							\space &\none[\dots] & l \\
							\none[\vdots]& \none[\cdots]\\
							\space &\none[\cdots]
						\end{ytableau}
						\renewcommand{\figurename}{Case}
						\renewcommand{\thefigure}{7}
						\captionof{figure}{}
					\end{minipage} 
					
				\end{tabular}
			\end{center}
			
			Let $ [\lambda_i'] $ and $ [\lambda_j'] $ denote the column on the left and the column on the right, respectively, in the pictures above.
			
			\noindent \textit{Case 1}: $ \lambda_i'=\lambda_j' $, $ l\in[\lambda_i'] $, $ l\in[\lambda_j'] $. By induction, $ [\lambda_j']\setminus\{l\} $ is obtained from $ [\lambda_i']\setminus\{l\} $ by changing at most one value. It follows that the same is true for $ [\lambda_i'] $ and $ [\lambda_j'] $.
			
			\noindent \textit{Case 2}: $ \lambda_i'=\lambda_j' $, $ l\notin[\lambda_i'] $, $ l\in[\lambda_j'] $. By induction, $ [\lambda_j']\setminus\{l\} $ is a subset of $ [\lambda_i'] $. It follows that $ [\lambda_j'] $ is obtained from $ [\lambda_i'] $ by changing its unique element not in $ [\lambda_j']\setminus\{l\} $ to $ l $.
			
			\noindent \textit{Case 3}: $ \lambda_i'=\lambda_j' $, $ l\notin[\lambda_i'] $, $ l\notin[\lambda_j'] $. By induction, $ [\lambda_j'] $ is obtained from $ [\lambda_i'] $ by changing at most one value.
			
			\noindent \textit{Case 4}: $ \lambda_i'>\lambda_j' $, $ l\in[\lambda_i'] $, $ l\in[\lambda_j'] $. By induction, $ [\lambda_j']\setminus\{l\} $ is a subset of $ [\lambda_i']\setminus\{l\} $. It follows that $ [\lambda_j'] $ is a subset of $ [\lambda_i'] $.
			
			\noindent \textit{Case 5}: $ \lambda_i'>\lambda_j' $, $ l\notin[\lambda_i'] $, $ l\notin[\lambda_j'] $. By induction, $ [\lambda_j'] $ is a subset of $ [\lambda_i']$.
			
			\noindent \textit{Case 6}: $ \lambda_i'>\lambda_j' $, $ l\in[\lambda_i'] $, $ l\notin[\lambda_j'] $. If $ \lambda_i'>\lambda_j'+1 $, then by induction, $ [\lambda_j'] $ is a subset of $ [\lambda_i']\setminus\{l\} $. It follows that $ [\lambda_j'] $ is a subset of $ [\lambda_i'] $.
			
			If $ \lambda_i'=\lambda_j'+1 $, then by induction, there are two possibilities: $ [\lambda_j']=[\lambda_i']\setminus\{l\} $ or $ [\lambda_j'] $ is obtained from $ [\lambda_i']\setminus\{l\} $ by changing one value. If $ [\lambda_j']=[\lambda_i']\setminus\{l\} $, then $ [\lambda_j'] $ is a subset of $ [\lambda_i'] $. We show that this is the only possibility.
			
			Otherwise, suppose that $ \lambda_i'=\lambda_j'+1 $ and $ [\lambda_j']\ne[\lambda_i']\setminus\{l\} $ for some $ i<j $. We first look at an example:
			
			\vspace*{10pt}
			
			\begin{center}
				\begin{tabular}{cc}
					
					\noindent\begin{minipage}{0.5\linewidth}
						\centering	
						\ytableausetup{mathmode,boxsize=2em}
						\begin{ytableau}
							\none[1] & \none[2] & \none[3] & \none[4] & \none[5] & \none[6] & \none[7] & \none[8] \\
							\color{red}1 &\color{red}1 & \color{red}1 &\color{red}1 &\color{red}1 &\color{red}1 &\color{yellow}2 &\color{violet}6\\
							\color{yellow}2 &\color{yellow}2 & \color{yellow}2 &\color{yellow}2 &\color{yellow}2 &\color{blue}3 & \color{blue}3 \\
							\color{blue}3 & \color{blue}3 &\color{blue}3 &\color{blue}3 &\color{green}4 &\color{green}4 &\color{green}4 \\
							\color{green}4 &\color{green}4 &\color{green}4 &\color{orange}5 &\color{orange}5 &\color{orange}5 &\color{orange}\circled 5 \\
							\color{orange}5 &\color{violet}\circled 6 &\color{violet}6 &\color{violet}6 &\color{violet}6 &\color{violet}6 &\color{violet}\circled 6\\
							\color{violet} 6 &\color{brown}\circled 7 &\color{brown}\circled 7\\
							\color{brown}7
						\end{ytableau}
						\setcounter{figure}{0}   
						\renewcommand{\figurename}{Tableau}
						\captionof{figure}{}
					\end{minipage} 
					
					$\rightarrow$
					
					\noindent\begin{minipage}{0.5\linewidth}
						\centering	
						\ytableausetup{mathmode,boxsize=2em}
						\begin{ytableau}
							\none[1] & \none[2] & \none[3] & \none[4] & \none[5] & \none[6] & \none[7] & \none[8] \\
							\color{red}1 &\color{red}1 & \color{red}1 &\color{red}1 &\color{red}1 &\color{red}1 &\color{yellow}2 &\color{violet}6\\
							\color{yellow}2 &\color{yellow}2 & \color{yellow}2 &\color{yellow}2 &\color{yellow}2 &\color{blue}3 & \color{blue}3 \\
							\color{blue}3 & \color{blue}3 &\color{blue}3 &\color{blue}3 &\color{green}4 &\color{green}4 &\color{green}4 \\
							\color{green}4 &\color{green}4 &\color{green}4 &\color{orange}5 &\color{orange}5 &\color{orange}5 &\color{violet}\circled 6 \\
							\color{orange}5 &\color{orange}\circled 5 &\color{violet}6 &\color{violet}6 &\color{violet}6 &\color{violet}6 &\color{brown}\circled 7\\
							\color{violet}6 &\color{violet}\circled 6 &\color{brown}\circled 7\\
							\color{brown}7
						\end{ytableau}
						\renewcommand{\figurename}{Tableau}
						\captionof{figure}{}
					\end{minipage} 
				\end{tabular}
			\end{center}
			
			It is clear that the above tableaux have the same shape and weight. Also, note that the Tableau 1 restricted to the shape $ (8,7,7,7,7,1,0) $ satisfies $ C_{8+7+7+7+7+1}' $.
			
			Let $ i=2 $, $ j=5 $, then $ \lambda_i'=6=\lambda_j'+1 $ and $ [\lambda_2']\setminus\{7\}\ne[\lambda_5'] $. Observe that moving the circled integers $(5,6,7,7,6)$ in Tableau 1 counterclockwise leads to Tableau 2. 
			
			To carry this out in general, let the $ s $-th column be the leftest column of length $ \lambda_i' $ such that $ T(\lambda_i',s)=l $. Let the $ S $-th column be the rightest column of length $ \lambda_j' $ such that $ T(\lambda_j',S)<l $. By our assumptions, both $ s $ and $ S $ exist. Moreover, $ s\le i $ and $ S\ge j $. Let the $ m $-th column be the rightest column of length $ \lambda_i' $. So, $ i\le m<j $. 
			
			If $ T(\lambda_j',s)<T(\lambda_j',S) $, then moving the integers 
			$$ (T(\lambda_j',S),\equalto{T(\lambda_i',s)}{l},\equalto{T(\lambda_i',s+1)}{l},...,\equalto{T(\lambda_i',m))}{l} $$
			in the tableau counterclockwise will generate a new tableau. One can easily check that we are simply interchanging $ T(\lambda_j',S) $ and $ T(\lambda_i',s)=l $ in this case. To see that this gives us a new tableau, it suffices to show that $ T(\lambda_i',s-1)\le T(\lambda_j',S) $ if $ s>1 $, as $ T(\lambda_j',S+1)=l $ if $ \lambda_{S+1}'=\lambda_{j}' $, ensures the validity of replacing $ T(\lambda_j',S) $ by $ l $. By induction, the definition of $ S $, and $  T(\lambda_j',s)<T(\lambda_j',S) $, we have $ [\lambda_S']\subsetneq[\lambda_{s-1}'] $, $ [\lambda_s']\setminus\{l\}\subsetneq[\lambda_{s-1}'] $ and $ \operatorname{Card}([\lambda_S']\cup([\lambda_s']\setminus\{l\}))\ge\lambda_i' $. However, $ \max\{[\lambda_S']\cup([\lambda_s']\setminus\{l\})\}=T(\lambda_j',S) $. It follows that $ T(\lambda_i',s-1)\le T(\lambda_j',S) $.
			
			In general, since $ [\lambda_j']\ne[\lambda_i']\setminus\{l\} $, there must exist a smallest $ h\in\{0,1,...,\lambda_j'-1\} $ such that $ T(\lambda_j'-h,s)<T(\lambda_j'-h,S) $. We have just shown that when $ h=0 $, doing the above counterclockwise shifting leads to a new tableau. For $ h>0 $, we will show that moving the integers:
			\begin{align*}
				( T(\lambda_j'-h,S), T(\lambda_j'-h+1,s), & T(\lambda_j'-h+2,s),..., T(\lambda_i',s),T(\lambda_i',s+1),...,\\
				& T(\lambda_i', m), T(\lambda_j',S),T(\lambda_j'-1,S),..., T(\lambda_j'-h+1, S) )
			\end{align*}
			in the tableau counterclockwise leads to a new tableau. 
			
			Note that $$ \{T(\lambda_i',s),T(\lambda_i',s+1),...,T(\lambda_i',m)\}=\{l\}, $$ it suffices to show that the new tableau is weakly increasing from the $ S $-th column to the $ (S+1) $-th column if there are more than $ S $ columns, and from the $ (s-1) $-th column to the $ s $-th column if $ s>1 $.

			By induction, $ [\lambda_{S+1}']\setminus\{l\}\subsetneq[\lambda_S'] $ if there are more than $ S $ columns. If $ \lambda_{S+1}'<\lambda_j'-h $, then we are done. Otherwise, we claim that $ T(\lambda_j'-h'+1,S)\le T(\lambda_j'-h',S+1) $ for all $ \max\{\lambda_j'-\lambda_{S+1}',1\}\le h'\le h $. Indeed, by induction, $ [\lambda_{S+1}']\setminus\{l\}\subsetneq[\lambda_S'] $, $ [\lambda_{S+1}']\setminus\{l\}\subsetneq[\lambda_{s}']\setminus\{l\} $, and by our assumption, $ [\lambda_{S}']\ne[\lambda_s']\setminus\{l\} $, and by the definition of $ h $, we have $ T(\lambda_j'-h,S)\notin[\lambda_s'] $ and $ T(\lambda_j'-h,S)>\max\{T(1,s),T(2,s),...,T(\lambda_j'-h,s)\} $. So, $ T(\lambda_j'-h,S)<T(\lambda_j'-h,S+1) $ and $ T(\lambda_j'-h,S+1)\in[\lambda_S']\cup\{l\} $. It follows that $ T(\lambda_j'-h+1,S)\le T(\lambda_j'-h,S+1) $. Since $ T(\lambda_j'-h+1,S+1)>T(\lambda_j'-h,S+1) $ and $ T(\lambda_j'-h+1,S+1)\in[\lambda_S']\cup\{l\} $, it follows that $ T(\lambda_j'-h+2,S)\le T(\lambda_j'-h+1,S+1) $. Similarly, $ T(\lambda_j'-h'+1,S)\le T(\lambda_j'-h',S+1) $ for all $ \max\{\lambda_j'-\lambda_{S+1}',1\}\le h'\le h $. And we have shown that the new tableau is indeed weakly increasing from the $ S $-th column to the $ (S+1) $-th column.
			
			Moreover, by definition we have $$ T(\lambda_j'-h,s)<T(\lambda_j'-h,S)<T(\lambda_j'-h+1,S)=T(\lambda_j'-h+1,s), $$ and the fact that $ [\lambda_S']\subsetneq[\lambda_{s-1}'] $, $ [\lambda_s']\setminus\{l\}\subsetneq[\lambda_{s-1}'] $, it follows that $ T(\lambda_j'-h+1,s-1)\le T(\lambda_j'-h,S) $. Similar to our argument above, one can see that $ T(\lambda_j'-h+h',s-1)\le T(\lambda_j'-h+h'-1,s) $ for all $ 2\le h'\le h+1 $. Hence, we have shown that the new tableau is weakly increasing from the $ (s-1) $-th column to the $ s $-th column.
			
			\noindent \textit{Case 7}: $ \lambda_i'>\lambda_j' $, $ l\notin[\lambda_i'] $, $ l\in[\lambda_j'] $. We show that this is impossible.
			
			Suppose not, by induction, we know that $ [\lambda_j']\setminus\{l\} \subset [\lambda_i'] $. Write $ [\lambda_j']=(c_1<c_2<...<c_m<l) $ and write $ [\lambda_i']\setminus [\lambda_j']=(x_1<x_2<...<x_{m'}) $, $m'>0 $. 
			
			We first study the case where $ j=i+1 $. If $ x_{m'}>c_m $, then interchanging $ x_{m'} $ and $ l $ gives a new tableau. Contradiction.
			
			Otherwise, $ x_{m'}<c_m $. If we take out $ x_{m'} $ from $ [\lambda_i'] $ with other integers in $ [\lambda_i'] $ undisturbed, move every integer below the previous $ x_{m'} $ one level up, move the $ l $ in $ [\lambda_j'] $ to the last box of $ [\lambda_i'] $, and finally insert $ x_{m'} $ into an appropriate position in $ [\lambda_j'] $. One can easily check that this new arrangement is compatible with the remaining tableau. And we have a new tableau. Contradiction.
			
			\begin{center}
				\begin{tabular}{cc}					
					\noindent\begin{minipage}{0.5\linewidth}
						\centering	
						\ytableausetup{mathmode,boxsize=2em}
						\begin{ytableau}
							\color{red}1 & \color{yellow}2\\
							\color{yellow}2 & \color{green}4\\
							\color{blue}3 & \color{violet}\circled 6\\
							\color{green}4 & \color{brown}\circled 7\\
							\color{orange}\circled 5\\
							\color{violet}\circled 6
						\end{ytableau}
						\renewcommand{\figurename}{Tableau}
						\captionof*{figure}{An example where $ j=2=i+1 $, $ \{c_1,...,c_m\}=\{2,4,6\} $, $ \{x_1,...,x_{m'}\}=\{1,3,5\} $, $ l=7 $. Note that the $ 6 $ in the first column is three levels below the $ 6 $ in the second column. In general, $ x_{m'} $ is at least two levels below the least integer greater than $ x_{m'} $ in the second column.}
					\end{minipage}
					$\rightarrow$
					
					\noindent\begin{minipage}{0.4\linewidth}
						\centering	
						\vspace*{-42pt}\ytableausetup{mathmode,boxsize=2em}
						\begin{ytableau}
							\color{red}1 & \color{yellow}2\\
							\color{yellow}2 & \color{green}4\\
							\color{blue}3 & \color{orange}\circled 5\\
							\color{green}4 & \color{violet}\circled 6\\
							\color{violet}\circled 6\\
							\color{brown}\circled 7
						\end{ytableau}
						\renewcommand{\figurename}{Tableau}
						\captionof*{figure}{This tableau is obtained from the left by moving the circled integers $ (5,6,7,6) $ clockwise. }
					\end{minipage}
				\end{tabular}
			\end{center}
			
			Now we study the case where $ j>i+1 $. Proceed by doing induction on the distance $ (j-i) $ of two columns. Suppose we have shown that Case 7 is impossible for all pairs of $ (i,j) $ such that $ 0<j-i<q $ in any tableau of $ n $ boxes, then we can assume without loss of generality that the cardinality of the set $ \{\lambda_i',\lambda_{i+1}',...,\lambda_j'\} $ is two, the $ j $-th column is the leftest column of length $ \lambda_{j}' $ containing $ l $, and every column of length $ \lambda_i' $ to the right of the $ i $-th column contains $ l $. Again, we look at an example:
			
			\vspace*{10pt}
			
			\begin{center}
				\begin{tabular}{cc}
					
					\noindent\begin{minipage}{0.5\linewidth}
						\centering	
						\ytableausetup{mathmode,boxsize=2em}
						\begin{ytableau}
							\none[1] & \none[2] & \none[3] & \none[4]\\
							\color{red}1 &\color{yellow}2 &\color{blue}3 & \color{blue}3  \\
							\color{yellow}2 &\color{blue}3 &\color{green}4 &\color{orange}\circled 5\\
							\color{blue}3 &\color{green}4  &\color{orange} 5 &\color{violet}\circled 6\\
							\color{green}\circled 4 &\color{orange} 5 \\
							\color{orange}\circled 5 &\color{violet} 6 
						\end{ytableau}
						\renewcommand{\figurename}{Tableau}
						\captionof{figure}{}
					\end{minipage} 
					
					$\rightarrow$
					
					\noindent\begin{minipage}{0.5\linewidth}
						\centering	
						\ytableausetup{mathmode,boxsize=2em}
						\begin{ytableau}
							\none[1] & \none[2] & \none[3] & \none[4]\\
							\color{red}1 &\color{yellow}2 &\color{blue}3 & \color{blue}3  \\
							\color{yellow}2 &\color{blue}3 &\color{green}4 &\color{green}\circled 4\\
							\color{blue}3 &\color{green}4  &\color{orange} 5 &\color{orange}\circled 5\\
							\color{orange}\circled 5 &\color{orange} 5 \\
							\color{violet}\circled 6 &\color{violet} 6 
						\end{ytableau}
						\renewcommand{\figurename}{Tableau}
						\captionof{figure}{}
					\end{minipage} 
				\end{tabular}
			\end{center}
			
			In the above tableaux, $ i=1, j=4 $ and $ l=6 $. Every column to the right of column $ 1 $ is a subset of $ [\lambda_1']\cup\{6\} $. Tableau 4 is obtained from Tableau 3 by moving the circled integers $ (4,5,6,5) $ clockwise.
			
			To carry this out in general, if $ T(\lambda_i',i)>T(\lambda_j'-1,j) $, then interchange $ T(\lambda_j',j)=l $ and $ T(\lambda_i',i) $. By induction, the result is still a tableau. Contradiction. 
			
			Otherwise, set $ T(0,j)=0 $ and let $ h $ be the smallest integer such that $ T(\lambda_j'-h,j)<T(\lambda_i'-h+1,i) $. Then we move the integers: $$ (T(\lambda_j'-h+1,i),T(\lambda_j'-(h-1)+1,i),...,T(\lambda_j',i), T(\lambda_j',j),T(\lambda_j'-1,j),...,T(\lambda_j'-h+1,j))  $$ in the tableau clockwise. By induction, $ [\lambda_q']\setminus\{l\}\subsetneq [\lambda_{i+1}'] $ for all $ q\ge j $. It follows that our construction gives another tableau. Contradiction.
			
			Hence, Case 7 is impossible and we have proved the necessity.
			
			$(C_n')\implies (C_n):$ Prove by contradiction. Suppose that there exists a tableau $ T $ satisfying $ (C_n') $ but there is another tableau $ T'\ne T $ with the same shape $ \lambda $ and weight. 
			
			Group the columns of $ \lambda $ in terms of their lengths, say $ \lambda_1'=\cdots=\lambda_{g_1}'>\lambda_{g_1+1}'=\cdots=\lambda_{g_1+g_2}'>\cdots>\lambda_{g_1+\dots+g_{d-1}}'>\lambda_{g_1+\dots+g_{d-1}+1}'=\cdots=\lambda_{g_1+\dots+g_{d}}'=\lambda_{\lambda_1}' $. Consider the rectangular tableau $ T|_{(d)} $ formed by the last $ g_d $ columns of $ T $. Let $ \mu_{(d)} $ be the weight of $ T|_{(d)} $. We claim that $ \mu_{(d)} $ is also the weight of $ T'|_{(d)} $.
			
			Suppose not, there must exist an integer $ z $ that appears in $ T|_{(d)} $ more often than it does in $ T'|_{(d)} $, as $ T|_{(d)} $ and $ T'|_{(d)} $ have the same shape. Since $ T $ satisfies $ (C_n') $, $ z $ also appears in the first $ \sum_{i=1}^{d-1}g_i $ columns. But there are at most $ \sum_{i=1}^{d-1}g_i $ $ z $'s in the first $ \sum_{i=1}^{d-1}g_i $ columns of $ T' $, since $ T' $ is also a tableau. Contradiction. We conclude that $ T|_{(d)} $ and $ T'|_{(d)} $ have the same weight $ \mu_{(d)} $.
			
			Similarly, we conclude that the $ i $-th rectangular tableaux $ T|_{(i)} $ and $ T'|_{(i)} $ have the same weight for all $ i=1,\cdots, d $.  And we conclude from this fact that all the positions of $ l $ are fixed, since $ l $ can only appear in the last few boxes of the last row of each rectangular tableau. If we consider the tableau without all the $ l $'s, then it is trivial to check that $ (C_{n-\norm{L}}') $ is satisfied by this smaller tableau with $ (n-\norm{L}) $ boxes. It follows that this smaller tableau is unique of its shape and weight. So, $ T=T' $, contradiction. Therefore, $ T $ is uniquely determined.

\end{proof}

\section{Proving factorization}
\label{sec:factorization}

\subsection{Three kinds of rectangular tableaux}
\label{sub:threekinds}

	A rectangular tableau $ T $ with $ m $ columns and $ k $ rows is unique of its shape and weight if and only if one of the following is true: 
\begin{itemize}
	\item[(1)] All columns of $ T $ are identical.(first kind)
	\item[(2)] The intersection of all columns of $ T $ has cardinality $ (k-1) $.(second kind)
	\item[(3)] $ T $ has $ (k+1) $ distinct integers in total.(third kind)
\end{itemize}
This follows directly from proposition \ref{prop: UniqueSST} and proposition \ref{prop: threekinds} whose proof is the following. Once theorem \ref{FactorizationFormula} is proved for each of the three kinds of the rectangular tableaux, the general case follows from doing induction on the number of rectangular blocks in $ \mu $.

\begin{proof}[Proof of Proposition \ref{prop: threekinds}] 
Deleting one of the repeated sets doesn't change kind. So we can without loss of generality assume that $(S_{1}, S_{2}, \ldots, S_{m})$ are all distinct. If $m=1$, it is the case of the first kind. If $m =2$, it is the case of both the second kind and the third kind. Suppose $m \geq 3$. $|S_{1} \cap S_{2}| =k-1$ and $|S_{1} \cup S_{2}| =k+1$. There are two cases: $S_{3} \not \subset S_{1} \cup S_{2}$ and $S_{3} \subset S_{1} \cup S_{2}$. Suppose $S_{3} \not \subset S_{1} \cup S_{2}$. Then $S_{1} \cap S_{2} \cap S_{3} =  S_{1} \backslash \{a\} = S_{2} \backslash \{b\} = S_{3} \backslash \{c\}$ for some distinct $a, b, c$, so $|S_{1} \cap S_{2} \cap S_{3}| = k-1$. If any other $S_i$ was to miss an element of $S_{1} \cap S_{2} \cap S_{3}$, then it would have to contain all $a, b, c$ and would have size $\geq (k-2)+3 > k$. Contradiction. So in this case  $\displaystyle \biggr \lvert \bigcap_{i=1}^{m} S_{i} \biggr \rvert = |S_{1} \cap S_{2} \cap S_{3}| = k-1$, hence it is the case of the second kind. Suppose now that $S_{3} \subset S_{1} \cup S_{2}$. If there was any other $S_{i} \not \subset S_{1} \cup S_{2}$, then we would arrive to contradiction as in the previous case by switching $S_{3}$ and $S_{i}$. So in this case any $S_{i} \subset S_{1} \cup S_{2}$, hence $\displaystyle \biggr \lvert \bigcup_{i=1}^{m} S_{i} \biggr \rvert = |S_{1} \cup S_{2}| = k+1$, hence it is the case of the third kind. 
\end{proof}

\subsection{Pieri formulas}
\label{sub:Pieri} 
Recall that $e_{\eta}$ for a partition $\eta$ denotes $\displaystyle \prod_{i=1}^{\infty} e_{\eta_{i}}$. 
Partitions admit dominance partial order: $\chi \unrhd \rho$ if $\chi_{1} + \chi_{2} + \cdots + \chi_{i} \geq \rho_{1} + \rho_{2} + \cdots + \rho_{i}$ for any $i \geq 1$. $\chi \rhd \rho$ means that $\chi \unrhd \rho$ and $\chi \neq \rho$. Repeated application of the vertical Pieri formulas \eqref{VerticalPieri} leads to 
\begin{align*}
e_{\mu'} = P_{\mu} + \sum_{\eta: \ \eta' \rhd \mu'} c_{\eta} P_{\eta}
\end{align*}
for some coefficients $c_{\eta}$. Therefore, we can write
\begin{align*}
	P_{\mu}=e_{\mu'}+\sum_{\eta: \ \eta' \rhd \mu'} (-c_{\eta}) P_{\eta}.
\end{align*} 
However, for any $ \eta\rhd\mu' $ we have
\begin{align*}
	P_{\eta}=e_{\eta'}+\sum_{\zeta: \ \zeta' \rhd \eta'} (-c_{\zeta}) P_{\zeta},
\end{align*} 
we can rewrite
\begin{align*}
	P_{\mu}=e_{\mu'}+\sum_{\eta: \ \eta' \rhd \mu'} (-c_{\eta})\left[ e_{\eta'}+\sum_{\zeta: \ \zeta' \rhd \eta'} (-c_{\zeta}) P_{\zeta}\right].
\end{align*} 
Continue substituting as above until the right side becomes a linear combination of elementary symmetric functions after finite steps. This yields

\begin{align}
\label{elementaryexpansion}
P_{\mu} = e_{\mu'} + \sum_{\eta: \ \eta' \rhd \mu'} d_{\eta'} e_{\eta'}
\end{align}
for some other coefficients $d_{\eta'}$. We say that $f \sim_{\lambda, \nu} g$ for symmetric functions $f, g$ if $P_{\lambda}(f-g)$ doesn't have $P_{\nu}$ in its expansion in the Macdonald basis. It follows from the vertical Pieri formulas \ref{VerticalPieri} that $f \sim_{\lambda, \nu} f + e_{\eta}$ whenever $\eta_{1}$ is larger than the number of non-zero rows of $\nu \backslash \lambda$. 

\subsection{Rectangular tableaux with identical columns}
\label{sub:identical}
Let $ \mu(=(M^K)) $ be the rectangular partition with $ M $ columns and $ K $ rows. Consider the case where the tableau $ T $ is on shape $ \mu $ with identical columns.

By formula \ref{elementaryexpansion}, we have
\begin{align*}
	P_{(M^K)}=e_K^M+\sum_{\eta:\eta_1>K}c_\eta e_\eta,
\end{align*}
where $ c_\eta $ are coefficients. From now on, we use ``$ (j,k,m)\text{ ad} $" to represent ``$ (j,k,m)\text{ admissible} $", use $ a_{j,k}^m(\text{respectively }\ b_{j,k}^m) $ for $ a(j,k,m)(\text{respectively }\ b(j,k,m)) $ for brevity. It follows that 
\begin{align*}
	c_{\lambda (M^K)}^\nu(q,t)&=\prod_{(j,k,m)\text{ ad}}\left(\dfrac{\frac{1-q^{\lambda_j-\lambda_k-m}t^{k-j+1}}{1-q^{\lambda_j-\lambda_k-m+1}t^{k-j}}}{\frac{1-q^{\lambda_j-\lambda_k-m}t^{k-j}}{1-q^{\lambda_j-\lambda_k-m+1}t^{k-j-1}}}\right)
	.
\end{align*}
On the other hand, theorem \ref{theorem: FactorizationFormula} claims that

\begin{align*}
	\tilde{c}_{\lambda (M^K)}^\nu(q,t)&=\psi_{T}(q,t)\prod_{(j,k,m)\text{ ad}}\left(\frac{X_{k} - q^{-a_{j,k}^m+b_{j,k}^m-1} t X_{j}}{X_{k} - q^{-a_{j,k}^m+b_{j,k}^m} \
		X_{j}}  \cdot \frac{X_{k} - q^{-a_{j,k}^m}t^{-1} X_{j}}{X_{k} - \
		q^{-a_{j,k}^m-1}X_{j}} \right)\\
	&=\prod_{(j, k, m) \text{ ad}}\left(\frac{1-q^{-a_{j,k}^m+b_{j,k}^m-1+\lambda_j-\lambda_k}t^{1+k-j}}{1-q^{-a_{j,k}^m+b_{j,k}^m+\lambda_j-\lambda_k}t^{k-j}}\cdot\frac{1-q^{-a_{j,k}^m+\lambda_j-\lambda_k}t^{-1+k-j}}{1-q^{-a_{j,k}^m-1+\lambda_j-\lambda_k}t^{k-j}} \right)\\
	&=\prod_{(j, k, m) \text{ ad}}\left(\frac{1-q^{-m+\lambda_j-\lambda_k}t^{1+k-j}}{1-q^{-m+1+\lambda_j-\lambda_k}t^{k-j}}\cdot\frac{1-q^{-m+1+\lambda_j-\lambda_k}t^{-1+k-j}}{1-q^{-m+\lambda_j-\lambda_k}t^{k-j}} \right).
\end{align*}

Thus, $ \tilde{c}_{\lambda (M^K)}^\nu(q,t)=c_{\lambda (M^K)}^\nu(q,t) $. It follows that theorem \ref{theorem: FactorizationFormula} is true for all rectangular tableaux with identical columns.

\subsection{Horizontal strips}
\label{sub:horizontal}
Let $ T $ be a horizontal row of $ n $ boxes with weight $ (0<d_1\le d_2\le \cdots \le d_n) $. Equivalently, we can write the weight as $ (1^{e_1},...,N^{e_N}) $, $ e_N>0 $.

By the horizontal Pieri formula \ref{HorizontalPieri}, we have 
	\begin{align*}
	P_\lambda\cdot P_{(r)}=\sum_{\lambda\prec_h \nu,\ |\nu|-|\lambda|=r} \frac{(q; q)_{r}}{(t; q)_{r}}\varphi_{\nu/\lambda}P_\nu, 
\end{align*}
where $ \nu/\lambda $ is a skew shape of $ r $ boxes without any two boxes in the same column, i.e., a horizontal strip,  and
$$ \varphi_{\nu/\lambda}:=\prod_{s\in C_{\nu/\lambda}}\frac{b_\nu(s)}{b_{\lambda}(s)}, \quad (a; q)_{s}:= \prod_{i=1}^{s} (1-aq^{i-1}).$$	

Meanwhile,
\begin{align}
	\label{6}
	P_\lambda P_{(n)}=\sum_{\nu}c_{\lambda (n)}^\nu(q,t) P_\nu.
\end{align}	

Since $ P_\nu $'s form a basis, we conclude that $ c_{\lambda(n)}^\nu(q,t) $ is not zero only when $ P_\nu $ appears in \eqref{6} and
\begin{align*}
	c_{\lambda(n)}^\nu(q,t)=\frac{\prod_{i=0}^{n-1}(1-q^{i+1})}{\prod_{i=0}^{n-1}(1-tq^{i})}\varphi_{\nu/\lambda}.
\end{align*}

By the definition of $ \varphi_{\nu/\lambda} $, we compute the Macdonald Littlewood-Richardson coefficient corresponding to the original partition $ \lambda $ and the horizontal tableau $ T $ as follows
\begin{align*}
	c_{\lambda(n)}^\nu(q,t)=&\frac{(q;q)_n}{(t;q)_n}\cdot\left(\prod_{i=1}^{N}\prod_{j=0}^{e_i-1}\frac{1-q^{j}t}{1-q^{j+1}}\right) \\
	&\cdot\prod_{\substack{i=1\\d_i>1}}^n\prod_{j=1}^{d_i-1}\frac{1-q^{\lambda_j+b(j,d_i,i)-(\lambda_{d_i}+a(j,d_i,i))-1}t^{d_i-j+1}}{1-q^{\lambda_j+b(j,d_i,i)-(\lambda_{d_i}+a(j,d_i,i))}t^{d_i-j}}\cdot \frac{1-q^{\lambda_j-(\lambda_{d_i}+a(j,d_i,i))}t^{d_i-j-1}}{1-q^{\lambda_j-(\lambda_{d_i}+a(j,d_i,i))-1}t^{d_i-j}}\\
	=&\frac{(q;q)_n}{(t;q)_n}\cdot\left(\prod_{i=1}^{N}\prod_{j=0}^{e_i-1}\frac{1-q^{j}t}{1-q^{j+1}}\right)\cdot\prod_{(j,k,m)\text{ ad}}\left(\frac{X_{k} - q^{-a_{j,k}^m+b_{j,k}^m-1} t X_{j}}{X_{k} - q^{-a_{j,k}^m+b_{j,k}^m} \
		X_{j}}  \cdot \frac{X_{k} - q^{-a_{j,k}^m}t^{-1} X_{j}}{X_{k} - \
		q^{-a_{j,k}^m-1}X_{j}} \right),
\end{align*}
where in the last line we substitute $ k $ for $ d_i $ and $ m $ for $ i $.

In theorem \ref{theorem: FactorizationFormula}, it is claimed that

\begin{align*}
	c_{\lambda(n)}^\nu(q,t)&=\psi_{T}(q,t)\prod_{(j,k,m)\text{ ad}}\left(\frac{X_{k} - q^{-a_{j,k}^m+b_{j,k}^m-1} t X_{j}}{X_{k} - q^{-a_{j,k}^m+b_{j,k}^m} \
		X_{j}}  \cdot \frac{X_{k} - q^{-a_{j,k}^m}t^{-1} X_{j}}{X_{k} - \
		q^{-a_{j,k}^m-1}X_{j}} \right).
\end{align*}
It suffices to show that
\begin{align*}
	\psi_{T}(q,t)=\frac{(q;q)_n}{(t;q)_n}\cdot\left(\prod_{i=1}^{N}\prod_{j=0}^{e_i-1}\frac{1-q^{j}t}{1-q^{j+1}}\right)=\frac{(q;q)_n}{(t;q)_n}\cdot\prod_{i=1}^{N}\frac{(t;q)_{e_i}}{(q;q)_{e_i}}.
\end{align*}

Notice that
\begin{align*}
	\psi_{T}(q,t)&=\prod_{i=1}^{N-1}\frac[10pt]{\frac{(t;q)_{\sum_{j=1}^{i}e_j}}{(q;q)_{\sum_{j=1}^{i}e_j}}}{\frac[5pt]{\frac{(t;q)_{\sum_{j=1}^{i+1}e_j}}{(t;q)_{e_{i+1}}}}{\frac{(q;q)_{\sum_{j=1}^{i+1}e_j}}{(q;q)_{e_{i+1}}}}}=\prod_{i=1}^{N-1}\frac{(t;q)_{\sum_{j=1}^{i}e_j}}{(q;q)_{\sum_{j=1}^{i}e_j}}\cdot\frac{(t;q)_{e_{i+1}}(q;q)_{\sum_{j=1}^{i+1}e_j}}{(t;q)_{\sum_{j=1}^{i+1}e_j}(q;q)_{e_{i+1}}}.
\end{align*}
One can easily check that
\begin{align*}
	\prod_{i=1}^{N-1}\frac{(t;q)_{\sum_{j=1}^{i}e_j}}{(q;q)_{\sum_{j=1}^{i}e_j}}\cdot\frac{(t;q)_{e_{i+1}}(q;q)_{\sum_{j=1}^{i+1}e_j}}{(t;q)_{\sum_{j=1}^{i+1}e_j}(q;q)_{e_{i+1}}}=\frac{(q;q)_n}{(t;q)_n}\cdot\prod_{i=1}^{N}\frac{(t;q)_{e_i}}{(q;q)_{e_i}}.
\end{align*}
It follows that theorem \ref{theorem: FactorizationFormula} holds for all tableaux with a single row.

\subsection{Rectangular blocks of the second kind}
\label{sub:identical}
Now we study rectangular tableaux $ T $ with $ m $ columns and $ k $ rows, and the intersection of weights of all columns has cardinality $ (k-1) $. 

For example, in the following tableau, $ \{ 1,3 \} $ is the intersection of weights of all columns.

\begin{center}
	\begin{minipage}{0.4\linewidth}
		\centering	
		\ytableausetup{mathmode,boxsize=2em}
		\begin{ytableau}
			1&1&1&1\\
			2&2&3&3\\
			3&3&4&5
		\end{ytableau}
		\renewcommand{\figurename}{Tableau}
		\captionof*{figure}{A rectangular tableau with $ m=4 $, $ k=3 $}
	\end{minipage}
\end{center}

\begin{definition}[Intrinsic] Suppose that $ T $ is an arbitrary rectangular tableau. If a box of $ T $ contains an integer that appears in every column of $ T $, then delete it and move every box below it one level up. The resulting rectangular tableau is called the intrinsic tableau of $ T $ and we denote it by $ \text{Int} (T) $.
	
\end{definition}

\vspace*{10pt}

\noindent\textbf{Example.}

$$
\begin{tabular}{ccc}
	\noindent\begin{minipage}{.3\linewidth}
		\centering
		\begin{ytableau}
			1&1&2\\
			2&2&3\\
			3&4&4
		\end{ytableau}
		\renewcommand{\figurename}{Tableau}
		\captionof*{figure}{$T$}
	\end{minipage} 
	
	$\Rightarrow$
	
	\noindent\begin{minipage}{.3\linewidth}
		\centering
		\begin{ytableau}
			1&1&3\\
			3&4&4\\
		\end{ytableau}
		\renewcommand{\figurename}{Tableau}
		\captionof*{figure}{$\text{Int}(T)$}
	\end{minipage} 
\end{tabular}
$$

\begin{claim}\label{cla:one}
	The homogeneous degree $ m $ component of $ P_{(m^k)}(1,...,1,x_{k},...,x_{k+m-1}) $ is equal to $ P_{(m)}(x_k,...,x_{k+m-1}) $.
\end{claim}

\begin{proof}
	Suppose that
	\begin{align*}
		P_{(m^k)}(x_1,...,x_{k+m-1})=\sum_{T}\psi_T(q,t)x_1^{T_1}\cdots x_{k+m-1}^{T_{k+m-1}},
	\end{align*}
	where $ T $ ranges over all semistandard tableaux on shape $ (m^k) $, and $ (1^{T_1},...,(k+m-1)^{T_{k+m-1}}) $ is the weight of tableau $ T $. Then we have
	\begin{align}
		P_{(m^k)}(1,...,1,x_{k},...,x_{k+m-1})=\sum_{T}\psi_T(q,t)x_k^{T_k}\cdots x_{k+m-1}^{T_{k+m-1}}.\label{homo m}
	\end{align}
	Note that $ \sum_{i=k}^{k+m-1}T_i=m $ implies $ \sum_{i=1}^{k-1}T_i=m(k-1) $. Moreover, since $ T_i\le m $ for all $ i $, we know that $ T_1=T_2=\cdots=T_{k-1}=m $ in order to have degree $ m $ in the right side of \eqref{homo m}. Hence, the intersection of weights of all columns of $ T $ corresponding to terms of degree $ m $ in \eqref{homo m}, which is $ \{1,2,...,k-1\} $, has cardinality $ (k-1) $. It follows that the intrinsic tableau $ \text{Int}(T) $ of $ T $ must be a horizontal strip with weight $ (k^{T_k},...,(k+m-1)^{T_{k+m-1}}) $. By the definition of $ \psi(q,t) $, it is clear that $ \psi_{T}(q,t)=\psi_{\text{Int}(T)}(q,t) $. Therefore, we conclude that the homogeneous degree $ m $ component of $ P_{(m^k)}(1,...,1,x_{k},...,x_{k+m-1}) $ is equal to $ P_{(m)}(x_k,...,x_{k+m-1}) $.
\end{proof}

Suppose that 
\begin{align*}
	P_{(m^k)}(x_1,...,x_{k+m-1})=\sum_{\mu}c_\mu e_\mu(x_1,...,x_{k+m-1})
\end{align*}
is the elementary symmetric polynomial expansion of $ P_{(m^k)}(x_1,...,x_{k+m-1}) $. 

\begin{claim}\label{cla:-(k-1)}
	For each $ \mu\vdash (m^k) $ such that the least monomial degree in $ e_\mu(1,...,1,x_k,...,x_{k+m-1}) $ is $ m $, $ c_\mu $ is equal to the coefficient of $ e_{\hat{\mu}} $ in the elementary symmetric polynomial expansion of $ P_{(m)}(x_1,...,x_m) $, where $ \hat{\mu}=(\mu_1-(k-1),...,\mu_{m}-(k-1)) $.
\end{claim}

\begin{proof}
	By formula \ref{elementaryexpansion}, we may assume that $ \mu=(\mu_1\ge\mu_2\ge\cdots \ge \mu_m\ge 0) $. If the least monomial degree in $ e_\mu(1,...,1,x_k,...,x_{k+m-1}) $ is $ m $, then we claim that $ \mu_m\ge k-1 $, since otherwise the least monomial degree in $ e_\mu(1,...,1,x_k,...,x_{k+m-1}) $ would be no less than $ mk-[(m-1)(k-1)+(k-2)]=m+1 $. Thus, $ \hat{\mu}=(\mu_1-(k-1),...,\mu_{m}-(k-1)) $ is a valid expression. Notice that the least homogeneous degree component in $ e_\mu(1,...,1,x_k,...,x_{k+m-1}) $ is exactly $ e_{\hat{\mu}}(x_k,...,x_{k+m-1}) $, which is of degree $ m $. For a $ \mu\vdash mk $ such that it contains a term of degree $ m $ in the elementary symmetric polynomial expansion of $ P_{(m)}(1,...,1,x_k,...,x_{k+m-1}) $ implies that the least monomial degree in $ e_\mu(1,...,1,x_k,...,x_{k+m-1}) $ is $ m $. Thus, by claim \ref{cla:one} we know that 
	\begin{align*}
		P_{(m)}(x_k,...,x_{k+m-1})=\sum_\mu c_\mu e_{\hat\mu}(x_k,...,x_{k+m-1}),
	\end{align*}
	where the summation ranges over $ \mu\vdash (m^k) $ such that the least monomial degree in $$ e_\mu(1,...,1,x_k,...,x_{k+m-1}) $$ is $ m $. Since the map that sends $ \mu $ to $ \hat{\mu} $ is a bijection between the set of partitions of $ (m^k) $ with exactly $ m $ parts and each part being no less than $ (k-1) $ and the set of partitions of $ m $, the claim follows once we change the variables $ x_i\mapsto x_{i-(k-1)} $, $ k\le i\le k+m-1 $.
\end{proof}

\noindent \textbf{Vertical Pieri's formulas as Quadruples}

Notice that if we multiply $ P_\lambda $ by a sequence of elementary symmetric polynomials and apply the vertical Pieri's formula \ref{VerticalPieri} repeatedly, then the desired coefficient in the final expansion is completely determined by a set of quadruples such as $ (i,j,q_i,q_j)^\lambda $, $ i<j $, representing the hook whose arm lies in the $ i $-th row and whose leg ends in the $ j $-th row. $ q_i(\text{respectively }\  q_j) $ is the number of boxes already attached to the $i(\text{respectively }\  j)$-th row of the original tableau, before the vertical strip to be attached(to the $j$-th row). Hence, we may use quadruples to represent the corresponding $ \psi' $'s.

Note that two quadruples $ (i,j,q_i,q_j)^\lambda $ and $ (I,J,q_I,q_J)^\lambda $ represent the same element in $$ \mathbb{Q}(q,t, X_1, X_2,...) $$ if and only if $ 0=i-I=j-J=(q_i-q_j)-(q_I-q_J) $. This can be easily verified by the definition of $ \psi' $. 

\begin{claim}\label{card k-1 claim}
	Theorem \ref{theorem: FactorizationFormula} holds for all rectangular $ T $ with $ m $ columns and $ k $ rows, and the intersection of weights of all columns has cardinality $ (k-1) $.
\end{claim}

\begin{proof}
	Suppose that $ r_1<\cdots< r_{k-1} $ are the common integers throughout all columns in $ T $, and $ \text{Int}(T) $ has weight $ (1^{\text{Int}(T)_1},...,l^{\text{Int}(T)_l}) $, where $ l $ is the largest integer in $ T $. Define the sets $ L_{\text{Int}}:=\{ i\in [l]:\text{Int}(T)_i>0 \} $; $ L_r:=\{r_1,...,r_{k-1}\} $; $ L_{0}:=[l]\setminus(L_{\text{Int}}\cup L_r) $.
	
	For tableau $ T $, by extra quadruples $ (i,j,q_i,q_j) $ we mean quadruples satisfying $ \{i,j\}\nsubseteq L_{\text{Int}} $ and $ (i,j)\notin L_0\times L_{\text{Int}} $.
	
	When we are multiplying $ P_{\lambda} $ by $ P_{(m^k)} $, we can write $ P_{(m^k)} $ in its elementary symmetric function expansion. Similarly, we can do this for $ P_{(m)} $. By corollary \ref{cla:-(k-1)}, we are able to identify the common factors as well as comparing the unique factors in these two related processes of multiplying a sequence of elementary symmetric functions, or the two processes of attaching vertical strips. And it boils down to keeping track of the extra quadruples induced by $ e_\kappa $ from $ P_{(m^k)} $ and the corresponding $ e_{\hat{\kappa}} $ from $ P_{(m)} $, respectively since the quadruples are the same if the quadruples are not extra. Clearly, $ \sum_{i=1}^{l}\text{Int}(T)_i=m $. 
	
	For tableau $ T $, the extra quadruples $ (i,j,q_i,q_j) $ can be classified into the following disjoint classes
	\begin{itemize}
		\item[(1)] $ i\in L_{\text{Int}}, j\in L_r $:
		
		$$
		J_1:=\bigcup_{\substack{i\in L_{\text{Int}}\\ j\in L_r\\i<j}}\bigcup_{q_j=0}^{m-\text{Int}(T)_i-1}\{ (i,j,0,q_j) \}.
		$$
		
		\item[(2)] $ i\in L_0, j\in L_r $:
		
		$$
		J_2:=\bigcup_{\substack{i\in L_0\\ j\in L_r\\ i<j}}\bigcup_{q_j=0}^{m-1}\{ (i,j,0,q_j) \}.
		$$
	\end{itemize}
	
	For tableau $ \text{Int}(T) $, the extra quadruples $ (i,j,q_i,q_j) $ are:
	
	\begin{itemize}
		\item[(a)] $ i\in L_{r}, j\in L_{\text{Int}} $:
		$$
		K:=\bigcup_{\substack{i\in L_r\\
				j\in L_{\text{Int}}\\ i<j}}\bigcup_{q_j=0}^{\text{Int}(T)_j-1}\{ (i,j,0,q_j) \}.
		$$
	\end{itemize}
	
	Now we consider the extra admissible triples for tableau $ T $
	
	\begin{itemize}
		\item[(1')] $ i\in L_{\text{Int}}, j\in L_r $:
		$$
		L_1:=\bigcup_{\substack{i\in L_{\text{Int}}\\ j\in L_r\\ i<j}}\bigcup_{h\in [m]\setminus C(i)}\{ (i,j,h) \},
		$$
		where $ C(i) $ is the set of columns in $ T $ containing $ i $.
		
		\item[(2')] $ i\in L_0, j\in L_r $:
		$$
		L_2:=\bigcup_{\substack{i\in L_0\\
				j\in L_r\\ i<j}}\bigcup_{h=1}^{m}\{(i,j,h)\}.
		$$
		
	\end{itemize}
	
	Similarly, the extra admissible triples for tableau $ \text{Int}(T) $ are
	
	\begin{itemize}
		\item[(a')] $ i\in L_r, j\in L_{\text{Int}} $:
		$$
		Q_:=\bigcup_{\substack{i\in L_r\\
				j\in L_{\text{Int}}\\ i<j}}\bigcup_{h\in C(j)}\{ (i,j,h) \},
		$$
		where $ C(j) $ is the set of columns in $ \text{Int}(T) $ containing $ j $.
	\end{itemize}
	
	It is easy to see that $ J_1,J_2 $ and $ L_1,L_2 $ represent the same factors on the right side of \ref{theorem: FactorizationFormula}, respectively, and so do $ K $ and $ Q $. Since we have proved that theorem \ref{theorem: FactorizationFormula} holds for all horizontal strips, and $ \psi_T(q,t)=\psi_{\text{Int}(T)}(q,t) $, the claim follows.
\end{proof}

\subsection{Rectangular blocks of the third kind}
\label{sub:identical}
Consider the case where the rectangular tableau has $ m $ columns and $ k $ rows and there are $ (k+1) $ distinct integers in the tableau.

To get the Macdonald Littlewood-Richardson coefficient in this case, we adopt the idea of ``truncated" Macdonald polynomial expression: $ \hat P_{(m^k)}:=\sum_{\mu} c_{\mu'}e_{\mu'} $, where $ \mu' $ is the conjugate partition of $ \mu $, $ c_{\mu'} $ are coefficients and the summation ranges over all partitions of $ (m^k) $ with $ \mu_1'\le k+1 $, since all $ e_{\mu'} $'s with $ \mu_1'>k+1 $ will not contribute to the Macdonald Littlewood-Richardson coefficient in this case by vertical Pieri's formula \ref{VerticalPieri}. Formally, we do the following computation:

\vspace*{10pt}

If we change the variables $ x_i\mapsto y_{i}^{-1} $ for $ 1\le i\le k+1 $ and set $ x_i=0 $ for all $ i>k+1 $, then by formula \ref{elementaryexpansion} we have
\begin{align*}
	P_{(m^k)}(y_1^{-1},...,y_{k+1}^{-1})e_{k+1}^m(y_1,...,y_{k+1})&=\sum_\mu c_{\mu'}\prod_{j=1}^{\ell(\mu')}e_{k+1-\mu'_{j}}(y_1,...,y_{k+1})\\
	&=\sum_{\mu}c_{\mu'}e_{\tilde{\mu'}}(y_1,...,y_{k+1}),
\end{align*}
where $ l(\mu') $ is the number of parts in $ \mu' $, and $ \tilde{\mu'} $ denotes the set-theoretic complement of $ \mu' $ in $ (m^{k+1}) $.

So, we get 
\begin{align*}
	P_{(m^k)}(y_1,...,y_{k+1})=e_{k+1}^m(y_1,...,y_{k+1})\sum_{\mu}c_{\mu'}e_{\tilde{\mu'}}(y_1^{-1},...,y_{k+1}^{-1}).
\end{align*}

On the other hand, if we consider the monomial expansion of $ P_{(m^k)} $ and do the same change of variables as above, we get
\begin{align}
	P_{(m^k)}(y_1^{-1},...,y_{k+1}^{-1})e_{k+1}^m(y_1,...,y_{k+1})&=\sum_{T}\psi_T(q,t)y_{i_m}\cdots y_{i_1},\label{4}
\end{align}
where the summation ranges over all semi-standard tableaux on shape $ (m^k) $ and $ y_{i_j} $ ($ y_{i_m}\ge \cdots\ge y_{i_1} $) represents the missing element of column $ (m-j+1) $. 

\vspace*{10pt}

\begin{claim}\label{cla:psi}
	The right side of equation (\ref{4}) is equal to $ P_{(m)}(y_1,...,y_{k+1}) $. 
\end{claim} 

\begin{proof}
	It suffices to show that $ \psi_T(q,t) $ is equal to the coefficient of $ y_{i_m}\cdots y_{i_1} $ in $ P_{(m)} $.
	
	Recall that(\cite{Mac99})
	\begin{align*}
		P_{\lambda/\mu}=\sum_{T}\psi_T(q,t)x^T,\qquad \psi_T(q,t)=\prod_{i=1}^r\psi_{\lambda^{(i)}/\lambda^{(i-1)}}(q,t),	
	\end{align*}
	where $ \psi_{\lambda^{(i)}/\lambda^{(i-1)}} $ is the $ i $-th horizontal strip in $ T $. For a horizontal strip $ \lambda/\mu $, we have
	\begin{align*}
		\psi_{\lambda/\mu}=\prod_{s\in R_{\lambda/\mu}-C_{\lambda/\mu}}\frac{b_\mu(s)}{b_\lambda(s)},
	\end{align*}
	where 
	\begin{align*}
		b_\lambda(s)=\begin{cases}
			\dfrac{1-q^{a_\lambda(s)}t^{l_\lambda(s)+1}}{1-q^{a_\lambda(s)+1}t^{l_\lambda(s)}},& s\in\lambda;\\
			1,&\text{otherwise.}
		\end{cases}	
	\end{align*}
	
	\vspace*{20pt}
	
	\noindent
	\begin{center}
		\begin{minipage}{0.4\linewidth}
			\centering	
			\ytableausetup{mathmode,boxsize=2em}
			\begin{ytableau}
				1&1&1&2\\
				2&2&3&3\\
				3&3&4&4\\
				4&5&5&5\\
				6&6&6&6
			\end{ytableau}
			\renewcommand{\figurename}{Tableau}
			\captionof*{figure}{A rectangular tableau with $ m=4 $, $ k=5 $}
		\end{minipage}
	\end{center}
	
	By the definition of $ \psi_T(q,t) $, we can rename all distinct integers $ u_1<\cdots <u_h $ as $ 1<\cdots<h $ in $ T $ without changing the value of $ \psi_T(q,t) $. Particularly, in the case that we are considering, we may assume without loss of generality that an integer $ t $ can only appear either in the $ t $-th row or in the $ (t-1) $-th row. As a result, we only need to keep track of the number of boxes to the right of a particular box, since at each step of attaching a horizontal strip, there are no boxes below any box in the horizontal strip which implies that $ b_\lambda(s) $ is completely determined by the arm length.
	
	In the following computation, suppose that there are $ z_k $ $ y_{i_j} $'s such that $ y_{i_j}=k $. And we use the arm lengths to represent the corresponding $ b_\lambda(s) $.	Let $ ``a\sim b" $ denote
	$$ \prod_{a\le x\le b}\frac{1-q^{x}t}{1-q^{x+1}}.$$ 
	
	Compute that	
	\begin{align}
		\psi_T(q,t)=\prod_{j=1}^k\frac{0\sim (z_{j+1}-1)}{\sum_{l=1}^{j}z_l\sim (\sum_{l=1}^{j+1}z_l-1)},\label{8}
	\end{align}
	and
	\begin{align}
		\psi_{(y_{i_1},...,y_{i_m})}(q,t)=\prod_{j=1}^{k}\frac{0\sim(\sum_{l=1}^{j}z_l-1)}{z_{j+1}\sim(\sum_{l=1}^{j+1}z_l-1)}.\label{9}
	\end{align}
	
	We compare the corresponding terms in (\ref{8}) and (\ref{9}). If $ z_{j+1}\le \sum_{l=1}^jz_l-1 $, then canceling out the repeated terms in (\ref{9}), we get the corresponding term in (\ref{8}). Likewise, if $ z_{j+1}-1\ge\sum_{l=1}^{j}z_l $, then canceling out the repeated terms in (\ref{8}), we get the corresponding term in (\ref{9}). If neither happens, then $ \sum_{l=1}^{j}z_l=z_{j+1} $ and the corresponding terms in (\ref{8}) and (\ref{9}) are obviously equal.	
\end{proof}

\begin{corollary}\label{col:coeff}
	The coefficient of any $ e_\kappa $ in the expansion of $ P_{(m^k)}(y_1,...,y_{k+1}) $ is equal to the coefficient of the elementary polynomial of the set-theoretic complement shape of $ \kappa $ in the rectangle $ (m^{k+1}) $ in the expansion of $ P_{(m)}(y_1,...,y_{k+1}) $.
\end{corollary}

\begin{proof}
	In claim \ref{cla:psi}, we have proved that
	\[ P_{(m)}(y_1,...,y_{k+1})=P_{(m^k)}(y_1^{-1},...,y_{k+1}^{-1})e_{k+1}^m(y_1,...,y_{k+1}). \]
	With the same notations as claim \ref{cla:psi}, we have
	\begin{align}
		P_{(m)}(y_1,...,y_{k+1})&=\left(\sum_\mu c_{\mu'} e_{\mu'}(y_1^{-1},...,y_{k+1}^{-1})\right)\cdot e_{k+1}^m(y_1,...,y_{k+1})=\sum_\mu c_{\mu'}e_{\tilde{\mu'}}(y_1,...,y_{k+1}),\label{P(m)one}
	\end{align}
	where the summation is over all possible partitions that appear in the truncated elementary symmetric polynomial expansion of $ P_{(m^k)} $, and $ \tilde{\mu'} $ is the set-theoretic complement shape of $ \mu' $ in the rectangle $ (m^{k+1}) $. 
	
	On the other hand, we have the elementary symmetric polynomial expansion of $ P_{(m)}(y_1,...,y_{k+1}) $:
	\begin{align}
		P_{(m)}(y_1,...,y_{k+1})=\sum_{\mu}d_{\tilde{\mu'}} e_{\tilde{\mu'}}(y_1,..,y_{k+1}),\label{P(m)two}
	\end{align}
	where $ d_{\tilde{\mu'}} $ are coefficients, the summation ranges over all partitions $ \mu \vdash m\times k $ with $ \mu_1'\le k+1 $ due to the bijection between all such partitions and partitions of $ m $ given by taking the set-theoretic complement of each other in $ (m^{k+1}) $. Compare the coefficients of $ e_{\tilde{\mu'}} $ on right sides of \eqref{P(m)one} and \eqref{P(m)two}, we conclude that
	\begin{align*}
		d_{\tilde{\mu'}}=c_{\mu'},
	\end{align*}
	due to the fact that the set of elementary symmetric polynomials forms a basis of the space of symmetric polynomials. Thus, the corollary follows.		
\end{proof}

All rectangular tableaux $ T $ considered below are unique of its weight and shape with $ m $ columns and $ k $ rows.

\begin{definition}[Complement]
	If a rectangular tableau $ T $ has $ m $ columns and $ k $ rows with $ (k+1) $ distinct integers, then each column misses exactly one of the $ (k+1) $ integers. Suppose that $ l $ is the largest integer in $ T $, and the column $ j $ misses integer $ y_j $, $ 1\le j\le m $, then the horizontal tableau with weight $ (l+1-y_1,l+1-y_{2},...,l+1-y_m) $ is called the complement of $ T $. We denote it by $ T^C $.
\end{definition}

\begin{definition}[Reversal]
	Let $ \lambda=(\lambda_1,...,\lambda_n) $ be a partition. Suppose that a rectangular tableau $ T $ has $ m $ columns and $ k $ rows whose largest integer is $ l $. For any $ N>\lambda_1 $, define the reversal of $ \lambda $ relative to $ T $ and $ N $ to be $ \lambda^N(T)=(N-\lambda_l,...,N-\lambda_1) $.
\end{definition}

\begin{claim}\label{cla:c_reduce}
	If the $ (k+1) $ distinct integers in $ T $ are exactly: $ 1,2,...,k+1 $, then the Macdonald Littlewood-Richardson coefficients $ c(\lambda,T,q,t) $ and $ c(\lambda^N(T),T^C,q,t) $ are equal for any given $ N>\lambda_1 $.
\end{claim}

\begin{proof}
	By corollary \ref{col:coeff}, it suffices to find for each partition $ \kappa $ of $ (m^k) $, a bijection between each summand in the coefficient of $ P_{\nu(T)} $, from the expansion of $ P_\lambda\cdot P_{(m^k)} $, originating from successively applying the vertical Pieri's formula \ref{VerticalPieri} to the sequence of elementary symmetric polynomials: $ e_{\kappa_1},e_{\kappa_2},...,e_{\kappa_m} $, and each summand in the coefficient of $ P_{\nu(T^C)} $, from the expansion of $ P_{\lambda^N(T)}\cdot P_{(m)} $, originating from successively applying the vertical Pieri's formula \ref{VerticalPieri} to the sequence of elementary symmetric polynomials: $ e_{k+1-\kappa_1},e_{k+1-\kappa_2},...,e_{k+1-\kappa_m} $.
	
	Suppose that at step $ s\in [m] $, the vertical strip with length $ \kappa_s $ is missing rows of $ \lambda $ with indices $ v_1<\cdots<v_{l-\kappa_s} $. Let $ V_s:=\{v_1,...,v_{l-\kappa_s}\} $. Let $ l(=k+1) $ be the largest integer in $ T $. The set of quadruples induced by this vertical strip is
	$$ \mathcal{S}_1:=\{ (i,j,q_i,q_j)^\lambda:i\in V_s, j\in [l]\setminus V_s,i<j \} .$$
	
	On the other hand, at step $ s\in [m] $, suppose that $ e_{k+1-\kappa_s} $ is not missing rows of $ \lambda^N(T) $ with indices $ l+1-v_{l-\kappa_s}<\cdots<l+1-v_1 $. Let $ W_s:=\{ l+1-v_{l-\kappa_s},...,l+1-v_1 \} $. Then the set of quadruples induced by this vertical strip is
	$$ \mathcal{S}_2:=\{ (I,J,p_I,p_J)^{\lambda^N(T)}: I\in [l]\setminus W_s, J\in W_s, I<J  \} .$$
	
	Since $ (I,J,p_I,p_J)^{\lambda^N(T)} $ represents the factor
	
	\begin{align*}
		\frac{1-q^{\lambda^N(T)_I+p_I-(\lambda^N(T)_J+p_J+1)}t^{J-I+1}}{1-q^{\lambda^N(T)_I+p_I-(\lambda^N(T)_J+p_J)}t^{J-I}}&=\frac{1-q^{N-\lambda_{l-I+1}+p_I-(N-\lambda_{l-J+1}+p_J+1)}t^{J-I+1}}{1-q^{N-\lambda_{l-I+1}+p_I-(N-\lambda_{l-J+1}+p_J)}t^{J-I}}\\
		&=\frac{1-q^{\lambda_{l-J+1}-\lambda_{l-I+1}+p_I-p_J-1}t^{J-I+1}}{1-q^{\lambda_{l-J+1}-\lambda_{l-I+1}+p_I-p_J}t^{J-I}}\\
		(\text{change the variables } J=l+1-i, I=l+1-j)\qquad&=\frac{1-q^{\lambda_{i}-\lambda_{j}+p_{l+1-j}-p_{l+1-i}-1}t^{j-i+1}}{1-q^{\lambda_{i}-\lambda_{j}+p_{l+1-j}-p_{l+1-i}}t^{j-i}}\\
		(p_{l+1-i}=s-q_i,p_{l+1-j}=s-q_j)\qquad&=\frac{1-q^{\lambda_{i}-\lambda_{j}+q_i-q_j-1}t^{j-i+1}}{1-q^{\lambda_{i}-\lambda_{j}+q_i-q_j}t^{j-i}},
	\end{align*}
	which is the same factor represented by the quadruple $ (i,j,q_i,q_j)^\lambda $. Hence there is a bijection between $ \mathcal{S}_1 $ and $ \mathcal{S}_2 $. Since we have assumed that the $ (k+1) $ distinct integers in $ T $ are exactly $ 1,2,...,k+1 $, $ \mathcal{S}_2 $ can be viewed as the set of quadruples induced by a summand in the elementary symmetric polynomial expansion of $ P_{(m)} $ in the specified way.
	
	\vspace*{10pt}
	
	\noindent\textit{Remark: The reason that this bijection only works when $ l=k+1 $ is that $ \sum_{i=1}^m(k+1-\kappa_s)=m $ when $ l=k+1 $. If $ l>k+1 $, then $ \sum_{i=1}^m(l-\kappa_s)>m $ and we are not able to claim that ``$ \mathcal{S}_2 $ can be viewed as the set of quadruples induced by a summand in the elementary symmetric polynomial expansion of $ P_{(m)} $", since there would be more than $ m $ add-ons attached to $ \lambda^N(T) $, but each $ e_{\eta'} $ in the elementary symmetric polynomial expansion of $ P_{(m)} $ satisfies $ \eta'\vdash m $.}
\end{proof}

\begin{corollary}
	Theorem \ref{theorem: FactorizationFormula} holds for all rectangular tableaux $ T $ with $ m $ columns and $ k $ rows with $ (k+1) $ distinct integers.
\end{corollary}

\begin{proof}
	Suppose that the $ (m-i+1) $-th column of the tableau $ T $ misses integer $ y_i $, $ i\in [m] $, then claim \ref{cla:psi} shows that $ \psi_T(q,t)=\psi_{(y_{i_1},...,y_{i_m})}(q,t) $. Since $ P_{(m)} $ is a symmetric polynomial, it follows that $ \psi_{(y_{i_1},...,y_{i_m})}(q,t)=\psi_{(l+1-y_{i_m},...,l+1-y_{i_1})}(q,t) $. If we add rows of identical integers less than $ l $ that have never appeared in $ T $, such that this new tableau satisfies the premise of claim \ref{cla:c_reduce}, then $ i $ is in column $ c $ of $ T $ if and only if $ (l+1-i) $ is not in column $ c $ of $ T^C $ for any $ i\in [l] $. Denote this new tableau by $ T^\ast $. Moreover, each factor on the right of \ref{theorem: FactorizationFormula} corresponding to the admissible triple $ (i,j,t)^\lambda $ for $ T^\ast $ is the same as the factor corresponding to the admissible triple $ (l+1-j,l+1-i,t)^{\lambda^N(T)} $ for $ T^C $. It follows that we can keep track of those admissible triples or quadruples originated from attaching vertical stripes to $ \lambda $ exclusively without involving $ \lambda^N(T) $. We are interested in the additional and missing admissible triples of $ T $ compared to $ T^\ast $ as well as quadruples. Though we may state ``compared to $ T^C $" below, the reader should keep in mind that this is the same thing as ``compared to $ T^\ast $".
	
	Partition the set $ [l]=L_0\sqcup L_- $ where $ L_- $ is the set of integers that never appear in $ T $. Note that if $ L_-=\emptyset $, then we are done by claim \ref{cla:c_reduce}. Compared to $ T^C $, $ T $ has the additional admissible triples
	$$
	H_1:=\bigcup_{\substack{i\in L_-\\ j\in L_0\\ i<j}}\bigcup_{t\in[m]\setminus c(j)}\{(i,j,t)\},
	$$
	where $ c(j) $ is the set of columns of $ T $ where $ j $ does not appear. Note that if $ (i,j,t)\in H_1 $, then $$ a(i,j,t)=\begin{cases}
		t-1,&t<\min c(j)\\
		t-1-\text{Card}(c(j)),&t>\max c(j)
	\end{cases},\qquad b(i,j,t)=0. $$
	
	Meanwhile, compared to $ T^C $, $ T $ is missing the following admissible triples
	$$
	H_2:=\bigcup_{j\in L_-}\{(y_{i_{m+1-t}},j,t):t\in [m], y_{i_{m+1-t}}<j\},
	$$
	where $ y_{i_{m+1-t}} $ is the unique missing integer among the $ (k+1) $ distinct integers in column $ t $ of $ T $. Note that if $ (i,j,t)\in H_2 $, then
	
	$$
	a(i,j,t)=t-\min(c(i)),\quad b(i,j,t)=0.
	$$
	
	Note that, however, if we examine the quadruples induced in the process of attaching vertical strips, then compared to $ T^C $, $ T $ has the additional quadruples
	$$ G_1:=\bigcup_{\substack{i\in L_-\\j\in L_0\\i<j}}\bigcup_{q_j=0}^{m-1-\text{Card}(c(j))}\{(i,j,0,q_j)\} .$$
	
	Meanwhile, compared to $ T^C $, $ T $ is missing the following quadruples
	$$
	G_2:=\bigcup_{\substack{i\in L_0\\j\in L_-\\i<j}}\bigcup_{q_j=0}^{\text{Card}(c(i))-1}\{ (i,j,0,q_j) \}.
	$$
	It is easy to see that $ H_1 (\text{respectively }\ H_2)$ and $ G_1(\text{respectively }\ G_2) $ represent the same factors on the right side of \ref{theorem: FactorizationFormula}. We are through.
\end{proof}

\subsection{The general case}
\label{sub:general}
Suppose that the shape of the conjugate partition $ \mu' $ of partition $ \mu(=\mu_1\ge\mu_2\ge\cdots) $ is 
$$ (\underbrace{k_1,...,k_1}_{m_1},\underbrace{k_2,...,k_2}_{m_2}, \cdots\cdots,\underbrace{k_d,...,k_d}_{m_d}) .
$$ 
Let $ \mu^1 $ be the rectangular block with $ m_1 $ columns and $ k_1 $ rows, and let $ \mu^0=\mu\setminus\mu^1 $ be the set-theoretic complement of $ \mu^1 $ in $ \mu $.

\begin{claim}
	Theorem \ref{theorem: FactorizationFormula} is true in general.
\end{claim}

\begin{proof}
	We have already proved the case where the number of rectangular blocks in $ \mu $ is $ d=1 $ and we proceed by induction on $ d $.
	
	By formula \ref{elementaryexpansion}, Pieri's formulas \ref{prop: Pieri} and the idea of truncated expressions, we know that
	\begin{align*}
		P_{\mu^1}\cdot P_{\mu^0}=P_{\mu}+\sum_{\eta'_1>k_1}c_{\mu^0 \mu^1}^\eta(q,t) P_\eta
	\end{align*}
	which can be viewed as multiplying $ P_{\mu^{0}} $ by $ P_{\mu^{1}} $, where $ \mu^{1} $ has the longest column $ k_1 $. By applying formula \ref{elementaryexpansion} for the special case where the partition is $ \mu^{1} $, a rectangle, we have $ \eta_{1}'>k_1 $ in the summation above. We also note that $ \eta'\rhd \kappa' $ by formula \ref{elementaryexpansion}.
	
	So,
	\begin{align*}
		P_\mu=P_{\mu^1}\cdot P_{\mu^0}-\sum_{\eta'_1>k_1}c_{\mu^0 \mu^1}^\eta(q,t) P_\eta.
	\end{align*}
	
	It follows from proposition \ref{prop: UniqueSST} and the definition of $ \psi_T(q,t) $ that for a tableau $ T(\mu) $, unique of its shape and weight, $ \psi_T(q,t) $ is multiplicative in terms of the rectangular blocks of $ T(\mu) $: $$ \psi_{T(\mu)}(q,t)=\psi_{T(\mu^0)}\cdot\psi_{T(\mu^1)}(q,t) .$$

	We first compute the coefficient of $ P_\nu $ in $ P_\lambda\cdot P_{\mu^1}\cdot P_{\mu^0}=(P_\lambda\cdot P_{\mu^1})\cdot P_{\mu^0} $. We claim that the intermediate Macdonald polynomial $ P_\kappa $ from the product $ P_\lambda\cdot P_{\mu^1} $ must correspond to the tableau $ T(\mu^1) $ in order to get $ P_{\nu} $ in the expansion of $ P_\kappa\cdot P_{\mu^0} $. Otherwise, since there will be $ (m_1\times k_1) $ add-ons in total compared to $ \lambda $ after multiplying $ P_{\mu^1} $, by proposition \ref{prop: UniqueSST}, there must exist either an integer $ I $ contained in $ T(\mu^0) $ such that the $ I $-th row of $ \kappa $ has more than $ m_1 $'s add-ons compared to $ \lambda $, or an integer $ J $ contained in $ T(\mu^1)\setminus T(\mu^0) $ such that the $ J $-th row of $ \kappa $ has more than $ \text{Card}(J) $ add-ons compared to $ \lambda $, where $ \text{Card}(J) $ is the number of times that $ J $ has appeared in $ T(\mu^1) $, which is also equal to the number of times that $ J $ appears in $ T(\mu) $. We know that the latter case will not yield $ P_\nu $ in the end, but the former case contradicts formula \ref{elementaryexpansion}.
	
	Let the weight of tableau $ T(\mu^1) $ be $ (1^{n_1},2^{n_2},...,) $, where for any $ i\in\mathbb{Z}_{>0} $, $ n_i $ is the number of times that $ i $ has appeared in $ T(\mu^1) $.
	
	Therefore, by induction and the fact that $ \psi_T(q,t) $ is multiplicative in terms of rectangular blocks whenever $ T $ is unique of its shape and weight, we compute the coefficient $ \tilde{c}^{\nu}_{\lambda\mu} $ of $ P_\nu $ in the expansion of $ P_\lambda \cdot P_{\mu} $ induced by $ P_\lambda\cdot P_{\mu^1}\cdot P_{\mu^0} $:
	\begin{align*}
		\tilde{c}^{\nu}_{\lambda\mu}(q,t)=&\psi_{T(\mu^1)}(q,t)\prod_{\substack{(j,k,m)^{T(\mu^1)}\text{ ad}\\ m\le m_1}}\left(\frac{X_k-q^{-a_{j,k}^m+b_{j,k}^m-1}tX_j}{X_k-q^{-a_{j,k}^m+b_{j,k}^m}X_j}\cdot\frac{X_k-q^{a_{j,k}^m}t^{-1}X_j}{X_k-q^{a_{j,k}^m-1}X_j}\right)\cdot\\
		&\psi_{T(\mu^0)}(q,t)\prod_{\substack{(j,k,m)^{T(\mu^0)}\text{ ad}\\ m\le m_2+\cdots+m_d}}\left(\frac{X_k-q^{n_j-n_k-a_{j,k}^m+b_{j,k}^m-1}tX_j}{X_k-q^{n_j-n_k-a_{j,k}^m+b_{j,k}^m}X_j}\cdot\frac{X_k-q^{n_j-n_k+a_{j,k}^m}t^{-1}X_j}{X_k-q^{n_j-n_k+a_{j,k}^m-1}X_j}\right)\\
		=&\psi_{T(\mu)}(q,t)\prod_{(j,k,m)^{T(\mu)}\text{ ad}}\left(\frac{X_k-q^{-a_{j,k}^m+b_{j,k}^m-1}tX_j}{X_k-q^{-a_{j,k}^m+b_{j,k}^m}X_j}\cdot\frac{X_k-q^{a_{j,k}^m}t^{-1}X_j}{X_k-q^{a_{j,k}^m-1}X_j}\right).
	\end{align*}
	
	To conclude the proof, we argue that $ \tilde{c}^{\nu}_{\lambda\mu}(q,t)=c^{\nu}_{\lambda\mu}(q,t) $.
	
	Recall that 
	\begin{align*}
		P_\mu=P_{\mu^1}\cdot P_{\mu^0}-\sum_{\eta'_1>k_1}c_{\mu^0 \mu^1}^\eta(q,t) P_\eta.
	\end{align*}
	It suffices to show that there is no $ P_\nu $ in the Macdonald polynomial expansion of $ P_\lambda\cdot P_\eta $ for any $ \eta $ such that $ c_{\mu^0 \mu^1}^\eta(q,t)\ne 0 $ in the above identity.
	
	Indeed, since $ \eta_1'>k_1 $, and every $ \eta $ must contain $ \mu^1 $, we deduce that 
	\begin{equation} \sum_{i>m_1}\eta_i'<\sum_{j=2}^d k_jm_j .\label{contra}
	\end{equation}
	Notice that by formula \ref{elementaryexpansion}, we have
	$$
	P_{\eta}=\sum_{\kappa'\unrhd\eta'}c_{\kappa'}e_{\kappa'},
	$$
	where $ c_{\kappa'} $ are coefficients. In particular, $ \kappa'_1>k_1 $ for all $ \kappa' $, since $ \eta_{1}'>k_1 $. In order to get $ P_\nu $ in $ P_\lambda\cdot e_{\kappa'} $, we argue that
	$$ \sum_{i=m_1+\cdots+m_{d-1}+1}^{m_1+\cdots +m_{d}}\kappa_i'=k_d m_d. 
	$$
	By proposition \ref{prop: UniqueSST}, we know that every integer appearing in the rightmost block of $ \mu $ must also appear in every column with greater length in $ \mu $. Moreover, by the vertical Pieri's formula \ref{VerticalPieri}, $ \kappa_{i} $'s where $ m_1+\cdots+m_{d-1}+1\le i\le m_1+\cdots+m_d $ must be able to provide enough boxes so that the resulting shape is $ \nu $, this translates as
	$$ \sum_{i=m_1+\cdots+m_{d-1}+1}^{m_1+\cdots +m_{d}}\kappa_i'\ge k_d m_d. 
	$$
	On the other hand, $ \kappa'\unrhd \eta'\unrhd\mu' $, it follows that 
	$$ \sum_{i=m_1+\cdots+m_{d-1}+1}^{m_1+\cdots +m_{d}}\kappa_i'\le k_d m_d. 
	$$
	Thus, we know that the number of boxes that $ e_{\kappa'_i} $'s, $ m_1+\cdots+m_{d-1}+1\le i\le m_1+\cdots m_d $ add to $ \lambda $ in each row is exactly given by the weight of the rightmost block $ (m_d^{k_d}) $ of $ \mu $.
	
	Likewise, one can argue successively that for any $ 2\le j\le d $, we have
	$$
	\sum_{i=m_1+\cdots+m_{j-1}+1}^{m_1+\cdots +m_{j}}\kappa_i'= k_j m_j.
	$$
	It follows that
	$$
	\sum_{i>m_1}\kappa_i'= \sum_{j=2}^dk_j m_j.
	$$
	Since $ \kappa'\unrhd\eta' $ and they are all partitions of $ |\mu| $, we know that
	$$
	\sum_{j=2}^dk_j m_j=\sum_{i>m_1}\kappa'_i\le\sum_{i>m_1}\eta'_i.
	$$
	However, this contradicts \eqref{contra}. Therefore, we conclude that $ \tilde{c}_{\lambda\mu}^\nu(q,t)=c_{\lambda\mu}^\nu(q,t) $.
	
\end{proof}

\section{Connection with Stanley's conjecture}\label{section:connection}

Recall that $$ U(a,l)=1-q^{a+1}t^{l} ,\quad L(a,l)=1-q^{a}t^{l+1} $$ are the upper hook length and the lower hook length, respectively. Then equation (\ref{theorem: FactorizationFormula}) can be rewritten in the following form
\begin{align}\label{eq:deformed}
	c_{\lambda\mu}^{\nu}(q,t)=\prod_{a,l,a',l'}\left(\frac{U(a,l)}{L(a,l)}\cdot\frac{L(a',l')}{U(a',l')}\right)
\end{align}
where the product is over wherever applicable. Note that one of the interesting features of \eqref{eq:deformed} is that we have the same number of fractions in the form $ U/L $ with those in the form $ L/U $.

Let $ T $ be a rectangular unique tableau of the second kind on shape $ \mu $ with $ m $ columns and $ k $ rows. For any integer $ z $, let $ c_z $ be the number of times that $ z $ has appeared in $ T $. Suppose that the corresponding Littlewood-Richardson coefficient $ c(\lambda,T)=1 $. Note that $ z\in\text{Int}(T) $ if and only if $ 0<c_z<m $. Also, $ c_{z}=m $ is equivalent to $ z $ appearing in every column of $ T $. We have the following lemma.

\begin{lemma}\label{lemma:characterization second kind}
 Under the above assumptions, assume that both $ j-1 $ and $ j $ have appeared as the second entry of some admissible triple of $ T $ and $ i,j\in\text{Int}(T) $ such that  $ i<j-1 $. Then 
 $$ \begin{cases}
 (\lambda_{j-1}+m)-(\lambda_{j}+c_{j})\ge c_{i}, &\text{when } c_{j-1}=m \\
 \lambda_{j-1}-(\lambda_{j}+c_{j})\ge 0, &\text{when } 0<c_{j-1}<m
 \end{cases}. 
$$
\end{lemma}

The following example illustrates this lemma.

\begin{center}
	\begin{tabular}{ccc}	
		
		\noindent\begin{minipage}{0.3\linewidth}
			\centering	
			\ytableausetup{mathmode,boxsize=1.3em}
				\begin{ytableau}
					\none&\none&\none&\none&\none&\none&\none&\none&\none&1&1\\
					\none&\none&\none&\none&\none&\none&\none&\none&1\\
					\none&\none&\none&\none&1&1&2&2&2\\
					\none&\none&\none&\none&2\\
					2&3&3&3&3\\
					3
			\end{ytableau}
		\end{minipage}
		
		$\quad \iff$

		\noindent\begin{minipage}{0.3\linewidth}
			\centering	
			\ytableausetup{mathmode,boxsize=1.3em}
			
				\begin{ytableau}
				1&1&2&3&3\\
				3&3&3&4&5\\
				5&5&5&5&6
			\end{ytableau}
		\end{minipage}
	\end{tabular}
	\renewcommand{\figurename}{Bijection}
	\captionof*{figure}{An example with $ c_1=2, c_2=1, c_3=5, c_4=1, c_5=5, c_6=1 $ and $ \text{Int}(T)={\tiny
			\begin{ytableau}
			1&1&2&4&6
		\end{ytableau}} $}
\end{center}

Assume that $ c(\lambda,T(\mu))=1>0 $ and $ T $ is unique of its shape and weight, then the Littlewood-Richardson tableau is uniquely determined by $ T(\mu) $. In particular, the weight of each row is exactly the same as if we assume rows of $ \lambda $ are far apart. Namely, from left to right in each row and from top to bottom, we first fill the boxes with integer $1$ until $ 1 $ is used up, then we fill the boxes with integer $ 2 $ until $ 2 $ is used up,......

\begin{proof}
	If $ j-1\in\text{Int}(T) $, then both the $ (j-1) $-th row and the $ j $-th row in $ \nu/\lambda $ are filled with the same integer. In order that $ c(\lambda,T)=1>0 $, there should be no overlap. We deduce that $ \lambda_{j-1}-(\lambda_j+c_j)\ge 0 $.
	
	If $ j-1\notin\text{Int}(T) $, then the last $ \sum_{t<j-1}c_{t} $ boxes in the $ (j-1) $-th row of $ \nu/\lambda $ and the $ j $-th row of $ \nu/\lambda $ are filled with the same integer. In order that $ c(\lambda,T)=1>0 $, there should be no overlap in the last $ \sum_{t<j-1}c_t $ boxes of the $ (j-1) $-th row with the $ j $-th row. We deduce that $ (\lambda_{j-1}+c_{j-1})-(\lambda_{j}+c_{j})\ge\sum_{t<j-1}c_t\ge c_{i} $.
\end{proof}

Let $ T $ be a rectangular unique tableau of the third kind on shape $ \mu $ with $ m $ columns and $ k $ rows. For any integer $ z $, let $ x_z $ be the number of columns of $ T $ that miss integer $ z $ among the total $ k+1 $ distinct integers in $ T $. Suppose that the corresponding Littlewood-Richardson coefficient $ c(\lambda,T)=1 $. Similarly, we have the following lemma.

\begin{lemma}\label{lemma:characterization third kind}
	Under the above assumptions, if both $ j-1 $ and $ j $ are in $ T $, then 
	$$ \lambda_{j-1}-\lambda_{j}\ge x_{j-1} .$$
\end{lemma}

The following example illustrates this lemma.

\begin{center}
	\begin{tabular}{ccc}	
		
		\noindent\begin{minipage}{0.3\linewidth}
			\centering	
			\ytableausetup{mathmode,boxsize=1.3em}
		 	\begin{ytableau}
					\none&\none&\none&\none&\none&1&1&1&1\\
					\none&\none&1&1&1&2&2\\
					2&2&2&2&2
			\end{ytableau}
		\end{minipage}
		
		$\quad \iff$

		\noindent\begin{minipage}{0.3\linewidth}
			\centering	
			\ytableausetup{mathmode,boxsize=1.3em}
		
				\begin{ytableau}
					1&1&1&1&2&2&2\\
					2&2&3&3&3&3&3
			\end{ytableau}
		\end{minipage}
	\end{tabular}
	\renewcommand{\figurename}{Bijection}
	\captionof*{figure}{An example with $ x_1=3,x_2=2,x_3=2 $ and $ \lambda_1-\lambda_2=3\ge x_1, \lambda_2-\lambda_3=2\ge x_2 $.}
\end{center}

\begin{proof}
	The number of integers that are less than $ j-1 $ in the $ (j-1) $-th row of $ \nu/\lambda $ is equal to 
	$$ \sum_{t=1}^{j-2}m-\sum_{t=1}^{j-2}(m-x_t)=\sum_{t=1}^{j-2}x_t.
	$$
	Notice that $ \sum_{t=1}^{j-2}x_t<m-x_{j-1} $, we conclude that the last $ m-\sum_{t=1}^{j-1}x_t $ boxes in the $ (j-1) $-th row are filled by integer $ j-1 $.
	
	Replacing $ j-1 $ by $ j $, we know that the number of integers that are less than $ j $ in the $ j $-th row of $ \nu/\lambda $ is equal to 
	$ \sum_{t=1}^{j-1}x_t $, and the last $  m-\sum_{t=1}^{j}x_t $ boxes are filled by $ j $. Since $ c(\lambda,T)=1>0 $, there should be no overlap in the last $ m-\sum_{t=1}^{j-1}x_t $ boxes of the $ (j-1) $-th row with the $ j $-th row. We deduce that
	\begin{align*} \sum_{t=1}^{j-1}x_t-(\lambda_{j-1}-\lambda_{j})\le \sum_{t=1}^{j-2}x_t	.	
	\end{align*}
	Hence, $
		\lambda_{j-1}-\lambda_{j}\ge x_{j-1}.$
\end{proof}
	
\noindent \textit{Remark: We have actually proved that each row of $ \nu/\lambda $ are filled by at most two consecutive integers under the assumptions of either Lemma \ref{lemma:characterization second kind} or Lemma \ref{lemma:characterization third kind}.}

\begin{proof}[Proof of Proposition \ref{prop:integral}]
	Since $ |\nu|=|\lambda|+|\mu| $, it suffices to show that each fraction in the form of $ U/L $(respectively, $ L/U $) in \eqref{eq:deformed} flips a hook in $ \lambda $ and $ \mu $(respectively, $ \nu $) from $ L $ to $ U $(respectively, from $ U $ to $ L $).
	
	First, we study the case where $ \mu $ is rectangular with $ m $ columns and $ k $ rows.
	
    Let $ T $ be a tableau on $ \mu $ which is unique of its shape and weight. Note that by the definition of $ \psi_T(q,t) $, Claim \ref{cla:psi} and Proposition \ref{prop: threekinds}, we can find a tableau $ H $ with continuous entries in $ 1<2<\cdots< u $, $u\le m$, with only one row and $ m $ columns, such that $ \psi_T(q,t)=\psi_{H}(q,t) $. 
    
    Let $ n_i $ be the number of integers less than or equal to $ i $ in $ H $. By the definition of $ \psi_H(q,t) $, we write
    \begin{align}\label{eq: First_Factor}
    	\begin{split}
    		\psi_T(q,t)=\psi_{H}(q,t)&=\prod_{i=2}^{u}\left(\prod_{j=0}^{n_{i-1}-1}\frac{L(j,0)}{U(j,0)}\cdot\prod_{j=n_{i}-n_{i-1}}^{n_{i}-1}\frac{U(j,0)}{L(j,0)}\right)\\
    		&=\left(\prod_{j=0}^{n_{1}-1}\frac{L(j,0)}{U(j,0)}\right)\cdot\left(\prod_{i=2}^{u-1}\prod_{j=0}^{n_{i}-n_{i-1}-1}\frac{L(j,0)}{U(j,0)}\right)\cdot\left( \prod_{j=n_u-n_{u-1}}^{n_u-1}\frac{U(j,0)}{L(j,0)} \right).
    	\end{split}
    \end{align}
    Note that $ n_1+\sum_{i=2}^{u-1}(n_i-n_{i-1})=n_{u-1}\le m. $ This suggests flip the boxes in the last row of $ \mu $ whose arm lengths equal $ n_u-n_{u-1},n_u-n_{u-1}+1,\cdots,n_u-1 $ and flip some boxes in $ \nu $ whose leg length is zero with corresponding arm lengths increasing from zero.
 
 	We also have
    
	\begin{align}\label{eq:second factor}
		\begin{split}
		&\prod_{(i,j,w) \ -\text{admissible}} \left(\frac{X_{j} - q^{-a(i,j,w)+b(i,j,w)-1} tX_{i}}{X_{j}-q^{-a(i,j, w)+b(i,j, w)} X_{i}}\cdot
		\frac{X_{j} - q^{-a(i,j,w)}t^{-1} X_{i}}{X_{j}-q^{-a(i,j,w)-1}X_{i}} \right)\\
		=& \prod_{(i,j,w) \ -\text{admissible}} \left(\frac{L(\lambda_{i}-\lambda_{j}-a_{i,j}^{w}+b_{i,j}^{w}-1,j-i)}{U(\lambda_{i}-\lambda_{j}-a_{i,j}^{w}+b_{i,j}^{w}-1,j-i)}\cdot
		\frac{U(\lambda_{i}-\lambda_{j}-a_{i,j}^{w}-1,j-i-1)}{L(\lambda_{i}-\lambda_{j}-a_{i,j}^{w}-1,j-i-1)} \right).
	\end{split}
	\end{align}

	If a pair $ (i,j) $ appears in an admissible triple, then by proposition \ref{prop: threekinds}, the associated $ b_{i,j}^{w} $ is the same for all admissible triples containing $ (i,j) $. Denote the number of admissible triples containing $ (i,j) $ by $ N_{i,j} $. Note that $ a_{i,j}^{w} $ continuously takes values $ 0,1,..., N_{i,j}-1. $ The first factor in the parenthesis of the expression \ref{eq:second factor} suggests flip the $ N_{i,j} $ consecutive boxes: $ (i,\nu_{j}-N_{i,j}+1),(i,\nu_{j}-N_{i,j}+2),\cdots,(i,\nu_{j}) $ in the $ i $-th row of $ \nu $. At the same time, the second factor suggests flip the $ N_{i,j} $ consecutive boxes: $ (i,\lambda_{j}+1), (i,\lambda_{j}+2),\cdots,(i,\lambda_{j}+N_{i,j}), $ in the $ i $-th row of $ \lambda $. In particular, we note that the leg lengths $ (j-i) $ appeared in the first factor is never equal to zero. This hints that the flips in $ \nu $ caused by factors from $ \psi_T(q,t) $ does not conflict with those caused by the first factor in \ref{eq:second factor}. 
	
	Indeed, by Lemma \ref{lemma:characterization second kind} and Lemma \ref{lemma:characterization third kind}, it is easy to see that we are guaranteed to flip the $ c_r $, when $ T $ is of the second kind(respectively, $ x_{r+1} $, when $ T $ is of the third kind.), boxes in the $ r $-th row(or boxes in rows directly under the $ r $-th row if there is any overlap) of $ \nu $ where $ c_r,x_r $ are defined below. We elaborate as follows:
	
	If $ T $ is of the second kind, and if every row are far apart, then obviously we are done. Otherwise, if $ 0<c_r, c_{r+1}<m $, then by the second inequality of Lemma \ref{lemma:characterization second kind}, we have $ c_{r} $ boxes with zero leg length. It remains to check for $ c_{r+1}=m $ while $ 0<c_{r}<m $. Since $ r+1 $ does not contribute factors in $ \psi_{T}(q,t) $, we can freely flip boxes in rows under the $ r $-th row with $ m $ add-ons until we reach the first row with less than $ m $ add-ons under the $ r $-th row. Denote this row by row $ f $. The first inequality of Lemma \ref{lemma:characterization second kind} tells us that $ (\lambda_{f-1}+m)-(\lambda_f+c_j)\ge c_r $ if we let $ i $ be $ r $. So, we can flip the boxes located in the lower right corners of those rows with $ m $ add-ons under the $ r $-th row.
	
	If $ T $ is of the third kind, and again, if every row are far apart, we are done. Otherwise, it follows from Lemma \ref{lemma:characterization third kind} that for $ r,r+1\in T $,
	$$
	(\lambda_{r}+m-x_r)-(\lambda_{r+1}+m-x_{r+1})\ge x_{r+1}
	$$
	which enables us to flip $ x_{r+1} $ boxes with zero leg length in the $ r $-th row of $ \nu $.
	
	If $ r+1\notin T $, and row $ f $ is the first row below the $ r $-th with at least one add-on and less than $ m $ add-ons(since if an integer appears throughout columns, then it will not contribute factors in $ \psi_T(q,t) $), then since $ x_f<m-x_{r} $, we are able to flip $ x_f $ boxes of zero leg length in the $ r $-th row of $ \nu $.
	
	Thus, we can keep our discussion of $ \psi_T(q,t) $ to a minimum from now on.
	
	Now we are ready to work out all the details based on the above ideas. It suffices to show that only those fractions in \ref{eq: First_Factor} and \ref{eq:second factor} that correspond to real arm and leg lengths combinations in $ \lambda,\mu,\nu $ will survive. We study the three kinds of rectangular tableaux.
	
	\noindent\textbullet $ T $ is of the first kind. Suppose that $ T $ is on shape $ \mu $ with $ m $ columns. Let $ i $ be an integer that has never appeared in $ T $. Then elements in $ \{(i,j,w):j\in T, j>i, w\in [m]\} $ are all admissible triples whose first entry is $ i $.
	
	Because of $ \psi_T(q,t)=1 $, we do not flip any box in $ \mu $.
	
	For boxes in $ \lambda $, we can naturally regard factors in the form of $ U/L $ in \ref{eq:second factor} as flippers of boxes from lower to upper. Unless the desired arm-leg lengths combination does not exist in $ \lambda $, then we call a fraction in \ref{eq:second factor} fictitious if there is no box in $ \lambda $ with the corresponding arm and leg lengths combination. A necessary condition for a fraction with leg length $ j-i-1 $ being fictitious is that there are add-ons in both the $ (j-1) $-th and the $ j $-th row. We can write down all the fictitious flippers of $ \lambda $ as
	\begin{align*}
		\left\{\frac{U(a,j-i-1)}{L(a,j-i-1)}:\lambda_i-\lambda_j-m\le a\le \lambda_i-\lambda_{j-1}-1  \right\}.
	\end{align*}
	Note that $ j-1 $ must also belong to $ T $ if there exists a fictitious flipper of $ \lambda $ whose corresponding admissible triple has $ (i,j) $ as the first two entries.

	However, the following factors whose leg lengths are $ j-1-i $ have appeared in \ref{eq:second factor}
	\begin{align*}
		\left\{\frac{L(a,j-1-i)}{U(a,j-1-i)}:\lambda_i-\lambda_{j-1}-m\le a\le \lambda_i-\lambda_{j-1}-1  \right\}.
	\end{align*}
	Since 
	\begin{align*}
		\lambda_i-\lambda_{j-1}-m\le\lambda_{i}-\lambda_{j}-m\qquad\text{ and }\qquad \lambda_i-\lambda_{j-1}-1=\lambda_i-\lambda_{j-1}-1, 
	\end{align*}
	we know that there is no fictitious flipper in the form of $ U/L $ in the reduced form of \ref{eq:second factor}.
	
	For boxes in $ \nu $, to argue that there is no fictitious flipper in the reduced form of $ L/U $ , we notice that the possible fictitious flippers are
	\begin{align*}
		\left\{\frac{L(a,j-1-i)}{U(a,j-1-i)}:\lambda_i-\lambda_{j}-m\le a\le \lambda_i-\lambda_{j-1}-1  \right\}.
	\end{align*}
	However, the following factors appear in \ref{eq:second factor}
	\begin{align*}
		\left\{\frac{U(a,j-i-1)}{L(a,j-i-1)}:\lambda_i-\lambda_j-m\le a\le \lambda_i-\lambda_j-1  \right\}.
	\end{align*}
	Since
	\begin{align*}
		\lambda_i-\lambda_j-m=\lambda_i-\lambda_j-m\qquad\text{ and }\qquad \lambda_{i}-\lambda_{j-1}-1\le\lambda_{i}-\lambda_{j}-1,
	\end{align*}
	we know that there is no fictitious flipper in the form of $ L/U $ in the reduced form of \ref{eq:second factor}. Thus, we proved the claim for $ \mu $ being the first kind rectangular tableau.
	
	\noindent\textbullet $ T $ is of the second kind. Suppose that $ T $ is on shape $ \mu $ with $ m $ columns. Let $ c_z $ be the number of times that $ z $ has appeared in $ T $.
	
	By the expression \ref{eq: First_Factor} and the discussion above, we flip the corresponding boxes in $ \mu $ and $ \nu $ without any further justification needed.
	
	For a pair $ (i,j) $ that has appeared as the first two entries in an admissible triple of $ T $. We discuss eight cases according to $ c_j=m $ or $ 0<c_j<m $; $ c_{j-1}=m $ or $ 0<c_{j-1}<m $; $ c_i=0 $ or $ 0<c_i<m $.
	
	\noindent(I) If $ c_i=0 $, all the fictitious $ U/L $ flippers are
	\begin{align*}
		\left\{\frac{U(a,j-i-1)}{L(a,j-i-1)}:\lambda_i-(\lambda_j+c_j)\le a\le \lambda_i-\lambda_{j-1}-1  \right\}.
	\end{align*}
	However, the following factors appear in \ref{eq:second factor}
	\begin{align*}
		\left\{\frac{L(a,j-1-i)}{U(a,j-1-i)}:\lambda_i-\lambda_{j-1}-c_{j-1}\le a\le \lambda_i-\lambda_{j-1}-1  \right\}.
	\end{align*}
	Since $ \lambda_j+c_j\le\lambda_{j-1}+c_{j-1} $, we have
	\begin{align*}
		\lambda_i-\lambda_{j-1}-c_{j-1}\le\lambda_i-(\lambda_j+c_{j})\qquad\text{ and }\qquad\lambda_i-\lambda_{j-1}-1= \lambda_i-\lambda_{j-1}-1,
	\end{align*}
	and we know that there is no fictitious flipper in the form of $ U/L $ in the reduced form of \ref{eq:second factor}.
	
	Meanwhile, all the fictitious $ L/U $ flippers are
	\begin{align*}	\left\{\frac{L(a,j-1-i)}{U(a,j-1-i)}:\lambda_i-(\lambda_{j}+c_j)\le a\le \lambda_i-\lambda_{j-1}-1  \right\}.
	\end{align*}
	However, the following factors appear in \ref{eq:second factor}
	\begin{align*}
		\left\{\frac{U(a,j-i-1)}{L(a,j-i-1)}:\lambda_i-\lambda_{j}-c_j\le a\le \lambda_i-\lambda_{j}-1  \right\}.
	\end{align*}
    Since
    \begin{align*}
    	\lambda_i-\lambda_j-c_j=\lambda_i-\lambda_j-c_j\qquad\text{ and }\qquad \lambda_i-\lambda_{j-1}-1\le\lambda_i-\lambda_j-1.
    \end{align*}
    We know that there is no fictitious flipper in the reduced form of \ref{eq:second factor}.
    
    \noindent(II) If $ 0<c_i<m, 0<c_{j-1}<m, 0<c_j<m $, all the fictitious $ U/L $ flippers are
    \begin{align*}
    	\left\{\frac{U(a,j-i-1)}{L(a,j-i-1)}:\lambda_i-\lambda_{j}-c_{j}\le a\le \lambda_i-\lambda_{j-1}-1 \right\}.
    \end{align*}
	However, the following factors appear in \ref{eq:second factor}
	\begin{align*}	
		\left\{\frac{L(a,j-1-i)}{U(a,j-1-i)}:\lambda_i-\lambda_{j-1}-c_{j-1}+c_i\le a\le \lambda_i-\lambda_{j-1}-1  \right\}.
	\end{align*}
	By Lemma \ref{lemma:characterization second kind}, we know that
	\begin{align*}
		\lambda_i-\lambda_{j-1}-c_{j-1}+c_i\le\lambda_i-\lambda_{j}-c_{j}\qquad\text{ and }\qquad\lambda_i-\lambda_{j-1}-1\le \lambda_i-\lambda_{j-1}+c_{j-1}-1 .
	\end{align*}	
	We know that there is no fictitious flipper in the form of $ U/L $.
	
	Meanwhile, the fictitious $ L/U $ flippers are
	\begin{align*}	
		\left\{\frac{L(a,j-1-i)}{U(a,j-1-i)}:\lambda_i-\lambda_{j}-c_{j}+c_i\le a\le \lambda_i+c_i-\lambda_{j-1}-1  \right\}.
	\end{align*}	
	However, the following factors appear in \ref{eq:second factor}
	\begin{align*}
		\left\{\frac{U(a,j-i-1)}{L(a,j-i-1)}:\lambda_i-\lambda_j-c_j\le a\le\lambda_i-\lambda_j-1 \right\}.
	\end{align*}
	By Lemma \ref{lemma:characterization second kind}, we know that
	\begin{align*}
		\lambda_i-\lambda_j-c_j\le\lambda_i-\lambda_j-c_j+c_i\qquad\text{ and }\qquad\lambda_i+c_i-\lambda_{j-1}-1 \le\lambda_i-\lambda_j-1.
	\end{align*}
	So, there is no fictitious flipper in the reduced form of \ref{eq:second factor}.
	
	\noindent(III) If $ 0<c_i<m $, $ 0<c_{j-1}<m, c_j=m $, all the fictitious $ U/L $ flippers are
	\begin{align*}
		\left\{\frac{U(a,j-i-1)}{L(a,j-i-1)}:\lambda_{i}-\lambda_{j}-(m-c_i)\le a\le\lambda_{i}-\lambda_{j-1}-1 \right\}.
	\end{align*}
	However, the following factors appear in \ref{eq:second factor}
	\begin{align*}	
		\left\{\frac{L(a,j-1-i)}{U(a,j-1-i)}:\lambda_i-\lambda_{j-1}-c_{j-1}+c_i\le a\le\lambda_i-\lambda_{j-1}+c_i-1  \right\}.
	\end{align*}
	Since $ \lambda_{j-1}+c_{j-1}\ge\lambda_j+m $, we have
	\begin{align*}
		\lambda_i-\lambda_{j-1}-c_{j-1}+c_i\le\lambda_{i}-\lambda_{j}-(m-c_i)\qquad\text{ and }\qquad\lambda_{i}-\lambda_{j-1}-1\le\lambda_i-\lambda_{j-1}+c_i-1 .
	\end{align*}
	We know that there is no fictitious flipper in the form of $ U/L $.
	
	Meanwhile, the fictitious $ L/U $ flippers are
	\begin{align*}	
		\left\{\frac{L(a,j-1-i)}{U(a,j-1-i)}:\lambda_i+c_i-\lambda_{j}-m\le a\le \lambda_i-\lambda_{j-1}-1  \right\}.
	\end{align*}
	However, the following factors appear in \ref{eq:second factor}
	\begin{align*}
		\left\{\frac{U(a,j-i-1)}{L(a,j-i-1)}:\lambda_{i}-\lambda_{j}-(m-c_i)\le a\le\lambda_{i}-\lambda_{j}-1 \right\}.
	\end{align*}
	Since 
	\begin{align*}
		\lambda_i+c_i-\lambda_j-m=\lambda_i-\lambda_j-(m-c_i)\qquad\text{ and }\qquad \lambda_i-\lambda_{j-1}-1\le\lambda_i-\lambda_j-1,
	\end{align*}
	we know that there is no fictitious flipper in the reduced form of \ref{eq:second factor}.

	\noindent(IV) If $ 0<c_i<m $, $ c_{j-1}=m $, all fictitious $ U/L $ flippers are
	
	\begin{align*}
		\left\{\frac{U(a,j-i-1)}{L(a,j-i-1)}:\lambda_{i}-\lambda_{j}-c_j\le a\le\lambda_i-\lambda_{j-1}-1 \right\}.
	\end{align*}
	However, the following factors appear in \ref{eq:second factor}
	\begin{align*}	
		\left\{\frac{L(a,j-1-i)}{U(a,j-1-i)}:\lambda_i-\lambda_{j-1}-(m-c_i)\le a\le \lambda_i-\lambda_{j-1}-1  \right\}.
	\end{align*}
	By Lemma \ref{lemma:characterization second kind}, we know that
	\begin{align*}
		\lambda_i-\lambda_{j-1}-(m-c_i)\le \lambda_i-\lambda_j-c_j\qquad\text{ and }\qquad\lambda_i-\lambda_{j-1}-1=\lambda_i-\lambda_{j-1}-1.
	\end{align*}
	We know that there is no fictitious flipper in the form of $ U/L $.
	
	Meanwhile, the fictitious $ L/U $ flippers are
	\begin{align*}
		\left\{\frac{L(a,j-1-i)}{U(a,j-1-i)}:\lambda_i+c_i-\lambda_j-c_j\le a\le\lambda_{i}-\lambda_{j-1}-1 \right\}.
	\end{align*}
	However, the following factors appear in \ref{eq:second factor}
	\begin{align*}
		\left\{\frac{U(a,j-i-1)}{L(a,j-i-1)}:\lambda_i-\lambda_j-c_j\le a\le\lambda_i-\lambda_j-1 \right\}.
	\end{align*}
	Since 
	\begin{align*}
		\lambda_i-\lambda_j-c_j\le\lambda_i+c_i-\lambda_j-c_j\qquad\text{ and }\qquad\lambda_i-\lambda_{j-1}-1\le\lambda_i-\lambda_j-1,
	\end{align*}
	we know that there is no fictitious flipper in the reduced form of \ref{eq:second factor}.	

\noindent\textbullet $ T $ is of the third kind. Suppose that $ T $ is on shape $ \mu $ with $ m $ columns. Let $ x_z $ be the number of columns of $ T $ that miss integer $ z $ among the total $ k+1 $ distinct integers in $ T $. Write $ (i,j,w) $ for any admissible triple. We discuss two cases according to whether $ i $ has ever appeared in $ T $.

\noindent(I) $ i $ has appeared in $ T $. We find that all the fictitious $ U/L $ flippers are

    \begin{align*}
         \left\{\frac{U(a,j-i-1)}{L(a,j-i-1)}:\lambda_{i}-\lambda_{j}-x_i\le a\le\lambda_i-\lambda_{j-1}-1 \right\}.
    \end{align*}
	However, the following factors appear in \ref{eq:second factor}
	\begin{align*}
	    \left\{\frac{L(a,j-1-i)}{U(a,j-1-i)}:\lambda_{i}-\lambda_{j-1}-x_{i}+x_{j-1}\le a\le\lambda_i-\lambda_{j-1}+x_{j-1}-1 \right\}.
    \end{align*}
	By Lemma \ref{lemma:characterization third kind}, we know that 
	\begin{align*}
		\lambda_{i}-\lambda_{j-1}-x_{i}+x_{j-1}\le\lambda_{i}-\lambda_{j}-x_i\qquad\text{ and }\qquad\lambda_i-\lambda_{j-1}-1\le\lambda_i-\lambda_{j-1}+x_{j-1}-1.
	\end{align*}
	And we know that there is no fictitious flipper in the form of $ U/L $.
	
	Meanwhile, the fictitious $ L/U $ flippers are
	\begin{align*}
		\left\{\frac{L(a,j-1-i)}{U(a,j-1-i)}:\lambda_i-\lambda_j-x_i+x_{j-1}\le a\le \lambda_i-\lambda_{j-1}+x_{j-1}-1 \right\}.
	\end{align*}
	However, the following factors appear in \ref{eq:second factor}
	\begin{align*}
		\left\{\frac{U(a,j-i-1)}{L(a,j-i-1)}:\lambda_i-\lambda_j-x_i\le a\le\lambda_i-\lambda_j-1 \right\}.
	\end{align*}
	By Lemma \ref{lemma:characterization third kind}, we know that
	\begin{align*}
		\lambda_i-\lambda_j-x_i\le\lambda_i-\lambda_j-x_i+x_{j-1}\qquad\text{ and }\qquad\lambda_i-\lambda_{j-1}+x_{j-1}-1\le\lambda_i-\lambda_j-1
	\end{align*}
	Thus, there is no fictitious flipper in the reduced form of \ref{eq:second factor}.
	
	\noindent(II) $ i $ has never appeared in $ T $. We find that all the fictitious $ U/L $ flippers are
	\begin{align*}
	    \left\{\frac{U(a,j-i-1)}{L(a,j-i-1)}:\lambda_i-\lambda_j-(m-x_i)\le a\le\lambda_i-\lambda_{j-1}-1 \right\}.
	\end{align*}	
	However, the following factors are in \ref{eq:second factor}
	\begin{align*}
		\left\{\frac{L(a,j-1-i)}{U(a,j-1-i)}:\lambda_i-\lambda_{j-1}-(m-x_i)\le a\le \lambda_i-\lambda_{j-1}-1 \right\}.
	\end{align*}
	It follows that
	\begin{align*}
		\lambda_i-\lambda_{j-1}-(m-x_i)\le\lambda_i-\lambda_j-(m-x_i)\qquad\text{ and }\qquad\lambda_i-\lambda_{j-1}-1=\lambda_i-\lambda_{j-1}-1.
	\end{align*}
	We know that there is no $ U/L $ flippers in the reduced form of \ref{eq:second factor}.
	
	Meanwhile, the fictitious $ L/U $ flippers are
	\begin{align*}
		\left\{\frac{L(a,j-1-i)}{U(a,j-1-i)}:\lambda_i-\lambda_{j}-(m-x_i)\le a\le \lambda_i-\lambda_{j-1}-1 \right\}.
	\end{align*}
	However, the following factors are in \ref{eq:second factor}
	\begin{align*}
		\left\{\frac{U(a,j-i-1)}{L(a,j-i-1)}:\lambda_i-\lambda_j-(m-x_i)\le a\le\lambda_i-\lambda_{j-1}-1 \right\}.
	\end{align*}
	So,
	\begin{align*}
    	\lambda_i-\lambda_{j}-(m-x_i)=\lambda_i-\lambda_j-(m-x_i)\qquad\text{ and }\qquad\lambda_i-\lambda_{j-1}-1=\lambda_i-\lambda_{j-1}-1.
	\end{align*}
	We conclude that there is no fictitious flipper in the reduced form of \ref{eq:second factor}.
	
	Now we study the general case where $ T $ that is not necessarily rectangular. 
	
	In order to show that the Proposition \ref{prop:integral} is true, we use induction on the number of vertical blocks in $ \mu $. 
	
	Since we have already settled the case where $ \mu $ only has one block, assuming that we have shown Proposition \ref{prop:integral} for all $ \mu $ with less than $ d $ blocks, it suffices to argue that Proposition \ref{prop:integral} is true for $ \mu $ with $ d $ blocks. Denote the rightest block in $ \mu $ by $ \mu_d $ and the set-theoretic complement of $ \mu_d $ in $ \mu $ by $ \mu^-_{d} $. Denote the corresponding unique tableau on shape $ \mu $(respectively, $ \mu_d $, $ \mu_{d}^- $) by $ T(\mu) $(respectively, $ T(\mu_d) $, $ T(\mu_d^-) $). Denote by $ \nu_{d}^- $ the larger partition corresponding to $ T(\mu_d^-) $. The induction hypothesis tells us that
	\begin{align}\label{eq:induction hypothesis}
		\left( \prod_{s\in\lambda}L(a(s),l(s)) \right)\cdot\left( \prod_{s\in\mu_d^-}L(a(s),l(s)) \right)\cdot\left( \prod_{s\in\nu_d^-}U(a(s),l(s)) \right)\cdot c_{\lambda\mu_d^-}^{\nu_d^-}(q,t)  
	\end{align}
	is a polynomial in $ L,U $ and the number of factors in the form of $ U $ equals the number of factors in the form of $ L $ and the last factor in \ref{eq:induction hypothesis} is assimilated into the first three factors as a sequence of flippers characterized by their arm and leg lengths. Compare \ref{eq:induction hypothesis} and \ref{prop:integral} as well as Theorem \ref{theorem: FactorizationFormula}, since we have pointed out in the proof of the general case of Theorem \ref{theorem: FactorizationFormula} that $ \psi_T(q,t) $ is multiplicative in terms of the rectangular blocks whenever $ T $ is unique of its shape and weight, we only need to manage to realize the remaining factors of \ref{eq:second factor} as flippers in $ \lambda $ and $ \nu $ and make sure that they do not clash with the existing flippers, and finally, migrate the existing flippers from $ \nu_{d}^- $ to $ \nu $.
	
	To work this out, we notice that there are two cases for an admissible triple $ (i,j,w) $ in $ T(\mu_d) $ determined by either $ i $ is in $ T(\mu_d) $.
	
	If $ i\in T(\mu_d) $, then both $ i,j $ appear in every column of $ T(\mu_d^-) $ which implies that $ i $ and $ j $ cannot appear as the first entry of any admissible in $ T(\mu_d^-) $, therefore the new flippers will not clash with the existing flippers. Moreover, both $ a(i,j,w)^{T(\mu_d)} $ and $ b(i,j,w)^{T(\mu_d)} $ take the same values as their counterparts $ a(i,j,w)^{T(\mu)} $ and $ b(i,j,w)^{T(\mu)} $, respectively. Thus, we can flip the hooks in $ \lambda,\nu $ as if we are doing this for a rectangular $ \mu $ discussed above, and migrate the existing hooks without any trouble.
	
	If $ i\notin T(\mu_d) $, then the $ i $-th row of $ \nu_{d}^- $ has the same length as the $ i $-th row $ \nu $. So, we do not have any trouble migrating the existing hooks in the $ i $-th row of $ \nu_{d}^- $. We note that $ b(i,j,w)^{T(\mu_d)}=b(i,j,w)^{T(\mu)} $. Moreover, since the smallest $ a(i,j,w)^{T(\mu)} $ is equal to the largest $ a(i,j,w)^{T(\mu_d^-)} $ plus one, we know that we can just keep flipping the corresponding hooks in the $ i $-th row of $ \nu $ as well as $ \lambda $ from where they are left towards the right. And this finishes the proof.
		
\end{proof}

\end{document}